\documentclass[11pt,a4paper]{article}

\usepackage{tikz}
\usepackage{amsthm}
\usepackage{amsmath}
\usepackage{amssymb}
\usepackage{mathrsfs}
\usepackage{geometry}
\usepackage{graphicx}
\usepackage{amsfonts}
\usepackage{epstopdf}
\usepackage{enumerate}
\usepackage{hyperref}
\usepackage{subcaption}

\usetikzlibrary{decorations.pathreplacing}

\allowdisplaybreaks

\newcommand{\vphi}{\varphi}

\newcommand{\md}{\mathrm{d}}
\newcommand{\loc}{{\mathrm{loc}}}

\newcommand{\R}{\mathbb{R}}

\newcommand{\rmnum}[1]{\romannumeral #1} 
\newcommand{\Rmnum}[1]{\uppercase\expandafter{\romannumeral#1}} 

\newcommand{\bbE}{\mathbb{E}}

\newcommand{\calC}{\mathcal{C}}
\newcommand{\calD}{\mathcal{D}}
\newcommand{\calE}{\mathcal{E}}
\newcommand{\calF}{\mathcal{F}}

\newcommand{\myset}[1]{\left\{#1\right\}}
\newcommand{\mybar}[1]{\overline{#1}}
\newcommand{\mynorm}[1]{\lVert#1\rVert}

\newtheorem{mythm}{Theorem}[section]
\newtheorem{myprop}[mythm]{Proposition}
\newtheorem{mylem}[mythm]{Lemma}
\newtheorem{mycor}[mythm]{Corollary}
\newtheorem{myrmk}[mythm]{Remark}
\newtheorem{myrmks}[mythm]{Remarks}

\newtheorem{myques}{Question}

\AtEndDocument{
\bigskip{\footnotesize%
  \textsc{Universit\'e Grenoble Alpes, CNRS UMR 5582, Institut Fourier, Gi\`eres, France.} \par
  \textit{E-mail address}: \texttt{yangmengqh@gmail.com}
}}

\begin{document}

\title{On the Domains of Dirichlet Forms on Metric Measure Spaces}
\author{Meng Yang}
\date{}

\maketitle

\abstract{We prove that the domain of the local Dirichlet form is \emph{strictly} contained in the domain of any stable-like non-local Dirichlet form on general metric measure spaces.}

\footnote{\textsl{Date}: \today}
\footnote{The author was very grateful to Dr. Eryan Hu for helpful discussions. The author was very grateful to anonymous suggestions to generalize to general metric measure spaces.}

\section{Introduction}

Let us recall the following classical Sobolev spaces.
\begin{align*}
H^1(\R^d)&=\myset{u\in L^2(\R^d):\int_{\R^d}|\nabla u(x)|^2\md x<+\infty},\\
H^\delta(\R^d)&=\myset{u\in L^2(\R^d):\int_{\R^d}\int_{\R^d}\frac{(u(x)-u(y))^2}{|x-y|^{d+2\delta}}\md x\md y<+\infty}\text{ for }\delta\in(0,1).
\end{align*}
We have their characterizations using Fourier transform as follows.
$$H^{\delta}(\R^d)=\myset{u\in L^2(\R^d):\int_{\R^d}|\calF[u](x)|^2|x|^{2\delta}\md x<+\infty}\text{ for }\delta\in(0,1],$$
where $\calF:L^2(\R^d)\to L^2(\R^d)$, $u\mapsto\calF[u]$ is the Fourier transform which is an isometry. It is easy to check that $H^\delta(\R^d)$ is monotone decreasing in $\delta\in(0,1]$, that is, $H^{\delta_1}(\R^d)\supseteq H^{\delta_2}(\R^d)$ for any $\delta_1,\delta_1\in(0,1]$ with $\delta_1<\delta_2$, hence $H^{\delta_0}(\R^d)\subseteq\cap_{\delta\in(0,\delta_0)}H^{\delta}(\R^d)$ for any $\delta_0\in(0,1]$. A natural question is whether this is equal. However, the answer is NO. Denote $\calF^{-1}:L^2(\R^d)\to L^2(\R^d)$ as the inverse Fourier transform. Let
$$u=\calF^{-1}\left[\left(x\mapsto\frac{1}{|x|^{\frac{d}{2}+\delta_0}}1_{|x|>1}\right)\right].$$
Then $u\in\cap_{\delta\in(0,\delta_0)}H^{\delta}(\R^d)\backslash H^{\delta_0}(\R^d)$ by the above characterizations.

The main purpose of this paper is to consider similar questions on general metric measure spaces where function spaces serve as the domains of Dirichlet forms.

Let $(K,d,\nu)$ be an $\alpha$-regular metric measure space. Consider the following non-local quadratic form
\begin{align*}
&\calE_\beta(u,u)=\int_K\int_K\frac{(u(x)-u(y))^2}{d(x,y)^{\alpha+\beta}}\nu(\md x)\nu(\md y),\\
&\calF_\beta=\myset{u\in L^2(K;\nu):\int_K\int_K\frac{(u(x)-u(y))^2}{d(x,y)^{\alpha+\beta}}\nu(\md x)\nu(\md y)<+\infty},
\end{align*}
where $\beta\in(0,+\infty)$ is so far arbitrary. Note that $\calF_\beta$ is monotone decreasing in $\beta$ and $\calF_\beta$ may be trivial for very large $\beta$. The critical exponent
$$\beta^*:=\sup\myset{\beta\in(0,+\infty):(\calE_\beta,\calF_\beta)\text{ is a regular Dirichlet form on }L^2(K;\nu)}$$
is called the walk dimension of the metric measure space $(K,d,\nu)$. It holds true on many metric measure spaces that $\beta^*\in(0,+\infty)$, $\calF_{\beta^*}$ is trivial but there lives a local regular Dirichlet form $(\calE_\loc,\calF_\loc)$ on $L^2(K;\nu)$ related to the critical exponent $\beta^*$. For example, the Euclidean spaces and a large family of fractal spaces including the Sierpi\'nski gasket and the Sierpi\'nski carpet. It is natural to raise the following questions.

\begin{myques}\label{ques_1}
Is it true that $\calF_\loc\subseteq\cap_{\beta\in(0,\beta^*)}\calF_\beta$? Is it true that $\calF_\loc=\cap_{\beta\in(0,\beta^*)}\calF_\beta$? 
\end{myques}

\begin{myques}\label{ques_2}
Given $\beta_0\in(0,\beta^*)$, is it true that $\calF_{\beta_0}=\cap_{\beta\in(0,\beta_0)}\calF_\beta$?
\end{myques}

On $\R^d$, $\beta^*=2$, $\calF_\loc=H^1(\R^d)$ and $\calF_{\beta}=H^{\beta/2}(\R^d)$ for $\beta\in(0,2)$. The answers to the above two questions were already given.

First, we consider two typical fractal spaces, that is, the Sierpi\'nski gasket and the Sierpi\'nski carpet. On the Sierpi\'nski gasket, $\beta^*=\log5/\log2$ and the existence of a local regular Dirichlet form was given by \cite{BP88,Kig89}. On the Sierpi\'nski carpet, $\beta^*\approx2.0969$ whose exact value is still unknown and the existence of a local regular Dirichlet form was given by \cite{BB89,KZ92}. We give the answer to Question \ref{ques_1} as follows.

\begin{myprop}\label{prop_SGSC}
On the Sierpi\'nski gasket (SG) and the Sierpi\'nski carpet (SC), we have
$$\calF_\loc\subsetneqq\bigcap_{\beta\in(0,\beta^*)}\calF_\beta.$$
\end{myprop}

\begin{myrmks}
\begin{enumerate}[(1)]
\item ``$\subseteq$" was used implicitly by Grigor'yan and the author \cite{GY18,GY19,Yan17} to show the denseness of $\calF_\beta$ in certain function spaces. It is relatively easy to prove ``$\subseteq$".
\item The key novelty of the above result is ``$\subsetneqq$". We will prove by \emph{constructing} explicit functions in $\cap_{\beta\in(0,\beta^*)}\calF_\beta\backslash\calF_\loc$. The construction on the SG is much easier than the construction on the SC due to the facts that the SG is a typical finitely ramified fractal space while the SC is a typical infinitely ramified fractal space. Due to the intrinsically analytic difference between the SG and the SC, the construction on the SG can \emph{not} be applied on the SC at all.
\item The above result can be in fact covered by the following Proposition \ref{prop_hk}. However, we will give the construction of some explicit functions in $\cap_{\beta\in(0,\beta^*)}\calF_\beta\backslash\calF_\loc$ in the proof instead of giving only the existence of such functions in the proof of Proposition \ref{prop_hk}.
\end{enumerate}
\end{myrmks}

Second, we give conditions to give answers to Question \ref{ques_1} and Question \ref{ques_2} as follows.

\begin{myprop}\label{prop_hk}
Let $(K,d,\nu)$ be an $\alpha$-regular metric measure space satisfying the chain condition and that all metric balls are relatively compact. Let $(\calE,\calF)$ be a conservative regular Dirichlet form on $L^2(K;\nu)$ with a heat kernel $p_t(x,y)$ satisfying
$$\frac{C_1}{t^{\alpha/\beta_0}}\Phi\left(C_2\frac{d(x,y)}{t^{1/\beta_0}}\right)\le p_t(x,y)\le\frac{C_3}{t^{\alpha/\beta_0}}\Phi\left(C_4\frac{d(x,y)}{t^{1/\beta_0}}\right)$$
for any $x,y\in K$, for any $t\in(0,\mathrm{diam}(K)^{\beta_0})$, here $\mathrm{diam}(K):=\sup\myset{d(x,y):x,y\in K}$ is infinite if $K$ is unbounded and is finite if $K$ is bounded, where $\beta_0\in(0,+\infty)$ is some parameter, $C_1,C_2,C_3,C_4$ are some positive constants and $\Phi:(0,+\infty)\to(0,+\infty)$ is some monotone decreasing function. Then
$$\calF\subsetneqq\bigcap_{\beta\in(0,\beta_0)}\calF_\beta.$$
\end{myprop}

\begin{myrmk}
Grigor'yan and Kumagai \cite[Theorem 4.1]{GK08} proved that under the above settings, the following dichotomy holds.
\begin{enumerate}[(a)]
\item Either $(\calE,\calF)$ on $L^2(K;\nu)$ is local, $\beta_0\in[2,\alpha+1]$ and
$$\Phi(s)\asymp C\exp\left(-cs^{\frac{\beta_0}{\beta_0-1}}\right).$$
\item Or $(\calE,\calF)$ on $L^2(K;\nu)$ is non-local, $\beta_0\in(0,\alpha+1]$ and
$$\Phi(s)\asymp(1+s)^{-(\alpha+\beta_0)}.$$
\end{enumerate}
If $(\calE,\calF)=(\calE_\loc,\calF_\loc)$ and $\beta_0=\beta^*$, then this is case (a) and the above result gives the answer to Question \ref{ques_1}. If $(\calE,\calF)=(\calE_{\beta_0},\calF_{\beta_0})$ and $\beta_0\in(0,\beta^*)$, then this is case (b) and the above result gives the answer to Question \ref{ques_2}.
\end{myrmk}

Proposition \ref{prop_hk} is indeed a consequence of the following result.

\begin{mythm}\label{thm_subordinate}
Let $(K,d,\nu)$ be a metric measure space. Let $(\calE,\calF)$ be a regular Dirichlet form on $L^2(K;\nu)$ that corresponds to a Hunt process $\myset{X_t}$. Let $A$ be the generator of $(\calE,\calF)$ on $L^2(K;\nu)$ which is a non-positive definite self-adjoint operator. For any $\delta\in(0,1)$, let $\myset{X^{(\delta)}_t}$ be the $\delta$-subordinated Hunt process that corresponds to a regular Dirichlet form $(\calE^{(\delta)},\calF^{(\delta)})$ on $L^2(K;\nu)$. Then
$$\calF\subseteq\bigcap_{\delta\in(0,1)}\calF^{(\delta)},$$
and ``$=$" if and only if $A$ is bounded.
\end{mythm}

\begin{myrmk}
``$\subseteq$" is also relatively easy to prove. We will prove ``$\subsetneqq$" when $A$ is unbounded by giving the \emph{existence} of some function in $\cap_{\delta\in(0,1)}\calF^{(\delta)}\backslash\calF$.
\end{myrmk}

This paper is organized as follows. In Section \ref{sec_SG}, we prove Proposition \ref{prop_SGSC} for the SG. In Section \ref{sec_SC}, we prove Proposition \ref{prop_SGSC} for the SC. In Section \ref{sec_subordinate}, we prove Theorem \ref{thm_subordinate}. In Section \ref{sec_hk}, we prove Proposition \ref{prop_hk}.

In this paper, we always assume that $(K,d,\nu)$ is a metric measure space, that is, $(K,d)$ is a locally compact separable metric space and $\nu$ is a Radon measure on $K$ with full support. We use $(\cdot,\cdot)$ to denote the inner product in $L^2(K;\nu)$. If $(\calE,\calF)$ is a closed form on $L^2(K;\nu)$, then we always define $\calE(u,u)=+\infty$ for any $u\in L^2(K;\nu)\backslash\calF$, hence $\calE$ is defined on the whole $L^2(K;\nu)$ rather than a dense subspace $\calF$.

NOTATION. The letters $c,C$ will always refer to some positive constants and may change at each occurrence. The sign $\asymp$ means that the ratio of the two sides is bounded from above and below by positive constants. The sign $\lesssim$ ($\gtrsim$) means that the LHS is bounded by positive constant times the RHS from above (below).

\section{Proof of Proposition \ref{prop_SGSC} for the SG}\label{sec_SG}

Let $K$ be the SG in $\R^2$, that is, let $p_0=0$, $p_1=(1,0)$, $p_2=(\frac{1}{2},\frac{\sqrt{3}}{2})$ and $f_i(x)=\frac{1}{2}(x+p_i)$, $x\in\R^2$, $i=0,1,2$, then $K$ is the unique non-empty compact set in $\R^2$ satisfying $K=\cup_{i=0}^2f_i(K)$, see Figure \ref{fig_SG}. Let $|\cdot|$ be the Euclidean metric in $\R^2$ and $\nu$ the normalized Hausdorff measure of dimension $\alpha=\frac{\log3}{\log2}$ on $K$. Then $(K,|\cdot|,\nu)$ is an $\alpha$-regular compact metric measure space.

\begin{figure}[ht]
\centering
\includegraphics[width=0.4\textwidth]{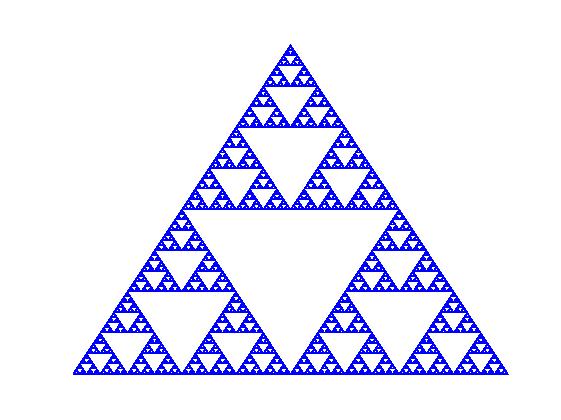}
\caption{The SG in $\R^2$}\label{fig_SG}
\end{figure}

Denote $l(S)$ as the set of all real-valued functions on a set $S$.

We list the characterization of local regular Dirichlet form on the SG as follows, see \cite{Kig89,Kig93,Kig01} for reference.

Let $W_0=\myset{\emptyset}$ and
$$W_n=\myset{w=w_1\ldots w_n:w_i=0,1,2,i=1,\ldots,n}\text{ for any }n\ge1.$$
For any $w^{(1)}=w^{(1)}_1\ldots w^{(1)}_m\in W_m$, $w^{(2)}=w^{(2)}_1\ldots w^{(2)}_n\in W_n$, let
$$w^{(1)}w^{(2)}=w^{(1)}_1\ldots w^{(1)}_mw^{(2)}_1\ldots w^{(2)}_n\in W_{m+n}.$$

Let $V_0=\myset{p_0,p_1,p_2}$ and $V_{n+1}=\cup_{i=0}^2f_i(V_n)$ for any $n\ge0$, then $\myset{V_n}_{n\ge1}$ is an increasing sequence of finite subsets of $K$ and the closure of $V_*=\cup_{n\ge1}V_n$ is $K$. For any $n\ge1$, for any $w=w_1\ldots w_n\in W_n$, let
$$V_w=f_{w_1}\circ\ldots\circ f_{w_n}(V_0),$$
then $V_n=\cup_{w\in W_n}V_w$.

For any $n\ge1$, for any $u\in l(V_n)$, let
$$a_n(u)=\left(\frac{5}{3}\right)^n\sum_{w\in W_n}\sum_{p,q\in V_w}(u(p)-u(q))^2,$$
then for any $m\le n$, for any $u\in l(V_n)$, we have $a_m(u)\le a_{n}(u)$.

We have $\beta^*=\frac{\log5}{\log2}$ and $(\calE_\loc,\calF_\loc)$ on $L^2(K;\nu)$ can be characterized as follows.
\begin{align*}
&\calE_\loc(u,u)=\lim_{n\to+\infty}a_n(u)=\sup_{n\ge1}a_n(u),\\
&\calF_\loc=\myset{u\in C(K):\sup_{n\ge1}a_n(u)<+\infty}.
\end{align*}

We have the characterization of the domain of stable-like non-local Dirichlet form as follows, see \cite[Theorem 1.1]{Yan18}. For any $\beta\in(\alpha,\beta^*)$, we have
$$\calF_\beta=\myset{u\in C(K):\sum_{n=1}^{+\infty}2^{(\beta-\beta^*)n}a_n(u)<+\infty}.$$

Therefore, it is obvious that $\calF_\loc\subseteq\cap_{\beta\in(0,\beta^*)}\calF_\beta$, to show that $\calF_\loc\subsetneqq\cap_{\beta\in(0,\beta^*)}\calF_\beta$, we only need to construct $u\in C(K)$ such that $a_n(u)=n$ for any $n\ge1$.

We construct $u\in l(V_*)$ by induction as follows.

For $n=1$. Let $u\in l(V_1)$ be given by $u=\frac{\sqrt{30}}{10}1_{\myset{p_0}}$, then $a_1(u)=1$.

Assume that we have constructed $u\in l(V_n)$ satisfying $a_n(u)=n$. Then for $n+1$, we only need to extend $u\in l(V_n)$ to a function on $V_{n+1}$ still denoted by $u\in l(V_{n+1})$.

Recall that
\begin{align*}
a_{n+1}(u)&=\left(\frac{5}{3}\right)^{n+1}\sum_{w\in W_{n+1}}\sum_{p,q\in V_w}(u(p)-u(q))^2\\
&=\left(\frac{5}{3}\right)^{n+1}\sum_{w\in W_{n}}\left(\sum_{i\in W}\sum_{p,q\in V_{wi}}(u(p)-u(q))^2\right).
\end{align*}

For any $w\in W_n$, we only need to assign the values of $u$ on $\cup_{i\in W}V_{wi}\backslash V_w$, see Figure \ref{fig_u}.

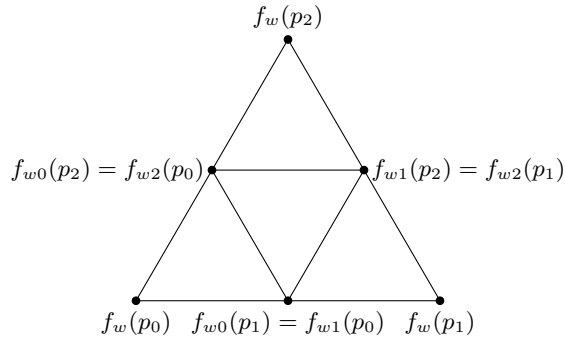
\begin{figure}[ht]
\centering
\begin{tikzpicture}[scale=1/4]

\draw (0,0)--(8,8*1.73205080756887)--(16,0)--cycle;
\draw (4,4*1.73205080756887)--(8,0)--(12,4*1.73205080756887)--cycle;

\draw[fill=black] (0,0) circle (0.2);
\draw[fill=black] (8,8*1.73205080756887) circle (0.2);
\draw[fill=black] (16,0) circle (0.2);
\draw[fill=black] (4,4*1.73205080756887) circle (0.2);
\draw[fill=black] (8,0) circle (0.2);
\draw[fill=black] (12,4*1.73205080756887) circle (0.2);

\draw (0,-1.2) node {{\footnotesize{$f_{w}(p_0)$}}};
\draw (16,-1.2) node {{\footnotesize{$f_{w}(p_1)$}}};
\draw (8,-1.2) node {{\footnotesize{$f_{w0}(p_1)=f_{w1}(p_0)$}}};
\draw (8,8*1.73205080756887+1.2) node {{\footnotesize{$f_{w}(p_2)$}}};
\draw (-1.5,4*1.73205080756887) node {{\footnotesize{$f_{w0}(p_2)=f_{w2}(p_0)$}}};
\draw (17.5,4*1.73205080756887) node {{\footnotesize{$f_{w1}(p_2)=f_{w2}(p_1)$}}};

\end{tikzpicture}

\caption{$\bigcup_{i\in W}V_{wi}$}\label{fig_u}
\end{figure}

Denote $a=u(f_w(p_0)),b=u(f_w(p_1)),c=u(f_w(p_2))$, let
\begin{align*}
u(f_{w1}(p_2))&=x=\alpha b+\alpha c+(1-2\alpha)a,\\
u(f_{w0}(p_2))&=y=\alpha c+\alpha a+(1-2\alpha)b,\\
u(f_{w0}(p_1))&=z=\alpha a+\alpha b+(1-2\alpha)c,
\end{align*}
where $\alpha\in(0,\frac{1}{2})$ is some parameter, then
\begin{align*}
&\sum_{i\in W}\sum_{p,q\in V_{wi}}(u(p)-u(q))^2\\
&=(a-z)^2+(b-z)^2+(a-y)^2+(c-y)^2+(b-x)^2+(c-x)^2\\
&+(x-y)^2+(y-z)^2+(z-x)^2\\
&=(15\alpha^2-12\alpha+3)\left((a-b)^2+(b-c)^2+(c-a)^2\right)\\
&=(15\alpha^2-12\alpha+3)\sum_{p,q\in V_w}(u(p)-u(q))^2.
\end{align*}
Let $\vphi(\alpha)=15\alpha^2-12\alpha+3$, $\alpha\in(0,\frac{1}{2})$, then
$$\min_{\alpha\in(0,\frac{1}{2})}\vphi(\alpha)=\vphi\left(\frac{2}{5}\right)=\frac{3}{5},\lim_{\alpha\downarrow0}\vphi(\alpha)=3,\lim_{\alpha\uparrow\frac{1}{2}}\vphi(\alpha)=\frac{3}{4},$$
hence there exists $\alpha_n\in(0,\frac{2}{5})$ such that $\vphi(\alpha_n)=\frac{3}{5}\frac{n+1}{n}$. Then we have the definition of $u$ on $\cup_{i\in W}V_{wi}$. Then we have the definition of $u$ on $V_{n+1}$. Moreover, $a_{n+1}(u)=\frac{n+1}{n}a_n(u)=n+1$.

By induction principle, we obtain $u\in l(V_*)$ satisfying $a_n(u)=n$ for any $n\ge1$. Since $\alpha_n\uparrow\frac{2}{5}$ as $n\to+\infty$, it is obvious that $u$ is uniformly continuous on $V_*$, hence $u$ can be extended to a continuous function on $K$ still denoted by $u\in C(K)$. Hence $u\in\cap_{\beta\in(0,\beta^*)}\calF_\beta\backslash\calF_\loc$.

\section{Proof of Proposition \ref{prop_SGSC} for the SC}\label{sec_SC}

Let $K$ be the SC in $\R^2$, that is, let
\begin{align*}
&p_0=(0,0),p_1=(\frac{1}{2},0),p_2=(1,0),p_3=(1,\frac{1}{2}),\\
&p_4=(1,1),p_5=(\frac{1}{2},1),p_6=(0,1),p_7=(0,\frac{1}{2}),
\end{align*}
and
$$f_i(x)=\frac{1}{3}(x-p_i)+p_i,x\in\R^2,i=0,1,\ldots,7,$$
then $K$ is the unique non-empty compact set in $\R^2$ satisfying $K=\cup_{i=0}^7f_i(K)$, see Figure \ref{fig_SC}. Let $|\cdot|$ be the Euclidean metric in $\R^2$ and $\nu$ the normalized Hausdorff measure of dimension $\alpha=\frac{\log8}{\log3}$ on $K$. Then $(K,|\cdot|,\nu)$ is an $\alpha$-regular compact metric measure space.

\begin{figure}[ht]
\centering
\includegraphics[width=0.74\textwidth]{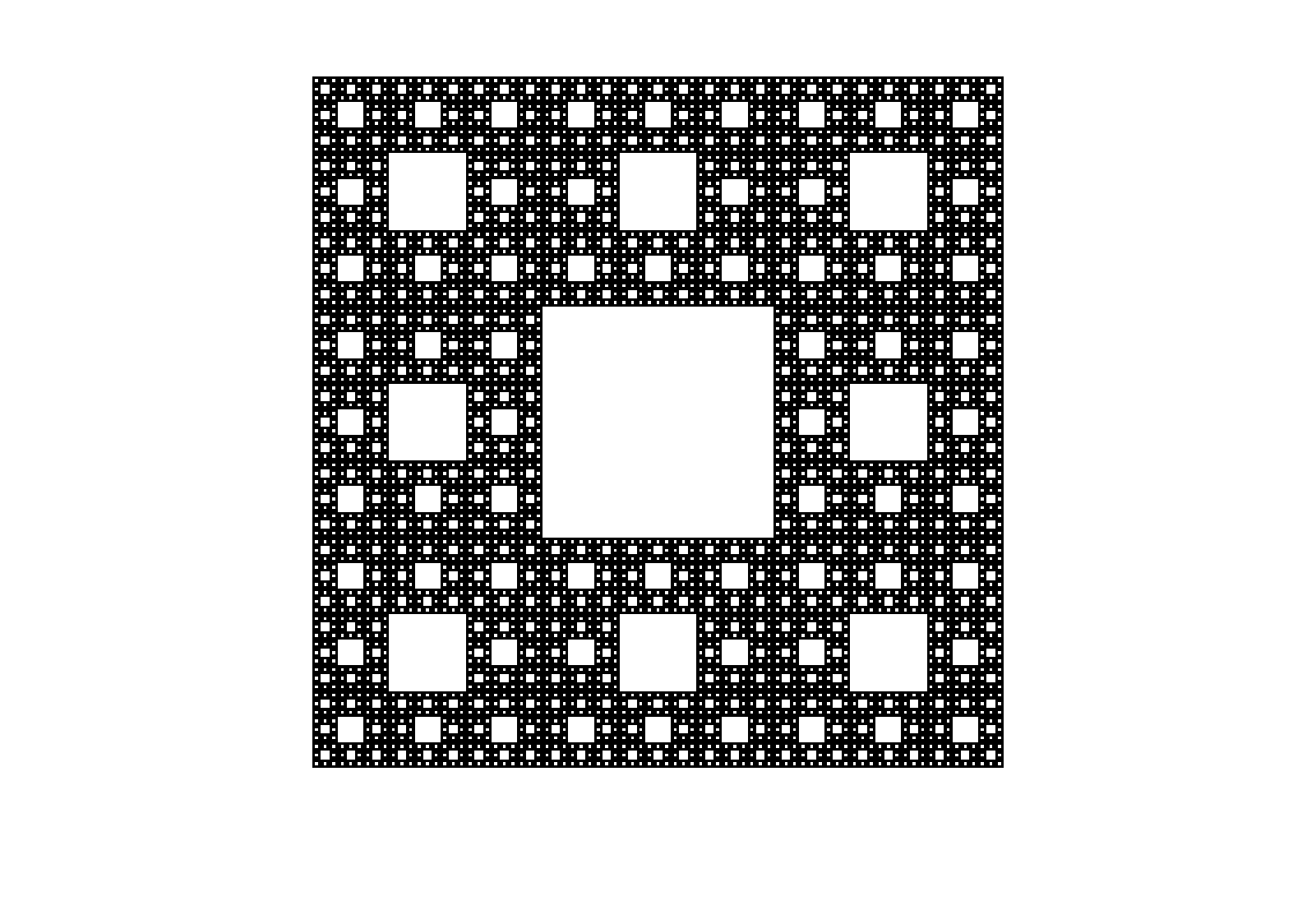}
\caption{The SC in $\R^2$}\label{fig_SC}
\end{figure}

Denote $l(S)$ as the set of all real-valued functions on a set $S$.

We list the characterization of local regular Dirichlet form on the SC from \cite{GY19} as follows, see also \cite{BB89,KZ92} for reference.

Let $W_0=\myset{\emptyset}$ and
$$W_n=\myset{w=w_1\ldots w_n:w_i=0,1,\ldots,7,i=1,\ldots,n}\text{ for any }n\ge1.$$
For any $w^{(1)}=w^{(1)}_1\ldots w^{(1)}_m\in W_m$, $w^{(2)}=w^{(2)}_1\ldots w^{(2)}_n\in W_n$, let
$$w^{(1)}w^{(2)}=w^{(1)}_1\ldots w^{(1)}_mw^{(2)}_1\ldots w^{(2)}_n\in W_{m+n}.$$
For any $i=0,1,\ldots,7$, let
$$i^n=\underbrace{i\ldots i}_{n\ \text{times}}\in W_n.$$

Let $V_0=\myset{p_0,p_1,\ldots,p_7}$ and $V_{n+1}=\cup_{i=0}^7f_i(V_n)$ for any $n\ge0$, then $\myset{V_n}_{n\ge0}$ is an increasing sequence of finite subsets of $K$ and the closure of $V_*=\cup_{n\ge0}V_n$ is $K$. For any $n\ge0$, for any $w=w_1\ldots w_n\in W_n$, let
$$f_w=f_{w_1}\circ\ldots\circ f_{w_n},$$
$$V_w=f_{w}(V_0),K_w=f_{w}(K),$$
then $V_n=\cup_{w\in W_n}V_w$. We use the convention that $f_\emptyset=\mathrm{Id}$ is the identity map.

Let $\rho$ be the parameter from resistance estimates in \cite{BB90,BBS90}, then
$$\beta^*=\frac{\log(8\rho)}{\log3}.$$
It was given in \cite{BB90,BBS90} that
\begin{itemize}
\item $\rho\in[\frac{7}{6},\frac{3}{2}]$ based on shorting and cutting technique,
\item $\rho\in[1.25147,1.25149]$ based on numerical calculation.
\end{itemize}
Hence $\rho>\frac{7}{6}$ which is a crucially important fact in the proof.

For any $n\ge0$, for any $u\in l(V_n)$, let
$$a_n(u)=\rho^n\sum_{w\in W_n}
{\sum_{\mbox{\tiny
$
\begin{subarray}{c}
p,q\in V_w\\
|p-q|=2^{-1}\cdot3^{-n}
\end{subarray}
$
}}}
(u(p)-u(q))^2.$$

By \cite[Theorem 2.5]{GY19}, $(\calE_\loc,\calF_\loc)$ on $L^2(K;\nu)$ can be characterized as follows.
\begin{align}
&\calE_\loc(u,u)\asymp\sup_{n\ge1}a_n(u),\nonumber\\
&\calF_\loc=\myset{u\in C(K):\sup_{n\ge1}a_n(u)<+\infty}.\label{eqn_domain_local}
\end{align}

We have the characterization of the domain of stable-like non-local Dirichlet form as follows, see \cite[Lemma 2.1]{GY19}. For any $\beta\in(\alpha,\beta^*)$, we have
\begin{equation}\label{eqn_domain_nonlocal}
\calF_\beta=\myset{u\in C(K):\sum_{n=1}^{+\infty}3^{(\beta-\beta^*)n}a_n(u)<+\infty}.
\end{equation}

Therefore, it is obvious that $\calF_\loc\subseteq\cap_{\beta\in(0,\beta^*)}\calF_\beta$, to show that $\calF_\loc\subsetneqq\cap_{\beta\in(0,\beta^*)}\calF_\beta$, we only need to prove the following result.

\begin{myprop}\label{prop_SC}
There exist $u\in C(K)$ and a universal positive constant $C$ satisfying $a_n(u)\le Cn$ for any $n\ge1$ and there exists some sequence $\myset{n_k}$ such that $a_{n_k}(u)\ge\frac{1}{C}n_k$ for any $k\ge1$.
\end{myprop}

\begin{proof}[Proof of Proposition \ref{prop_SGSC} using Proposition \ref{prop_SC}]
Since $a_n(u)\le Cn$ for any $n\ge1$, by Equation (\ref{eqn_domain_nonlocal}), we have $u\in\calF_\beta$ for any $\beta\in(\alpha,\beta^*)$. Since $a_{n_k}(u)\ge\frac{1}{C}n_k$ for any $k\ge1$, by Equation (\ref{eqn_domain_local}), we have $u\not\in\calF_\loc$. Hence $u\in\cap_{\beta\in(\alpha,\beta^*)}\calF_\beta\backslash\calF_\loc=\cap_{\beta\in(0,\beta^*)}\calF_\beta\backslash\calF_\loc$.
\end{proof}

Denote
$$L=\left(\myset{0,1}\times[0,1]\right)\bigcup\left([0,1]\times\myset{0,1}\right).$$

For any $n\ge0$, for any $u\in l(V_n)$, denote
$$E_n(u)=\sum_{w\in W_n}
{\sum_{\mbox{\tiny
$
\begin{subarray}{c}
p,q\in V_w\\
|p-q|=2^{-1}\cdot3^{-n}
\end{subarray}
$
}}}
(u(p)-u(q))^2.$$

We construct two functions in $C(K)$ that vanish on $L$ and have explicit asymptotic behaviours of $E_n$ as follows.

\begin{mylem}\label{lem_u1}
There exist $u_1\in C(K)$ and a universal positive constant $C$ satisfying $0\le u_1\le1$ on $K$, $u_1|_L=0$ and
$$\frac{1}{C}\left(\frac{1}{\rho}\right)^n\le E_n(u_1)\le C\left(\frac{1}{\rho}\right)^n\text{ for any }n\ge1.$$
\end{mylem}

\begin{mylem}\label{lem_u2}
There exist $u_2\in C(K)$ and a universal positive constant $C$ satisfying $0\le u_2\le1$ on $K$, $u_2|_L=0$ and
$$\frac{1}{C}\left(\frac{6}{7}\right)^n\le E_n(u_2)\le C\left(\frac{6}{7}\right)^n\text{ for any }n\ge1.$$
\end{mylem}

\begin{myrmk}
Since $u_1|_L=u_2|_L=0$, we have $E_0(u_1)=E_0(u_2)=0$. For any $n\ge1$, the lower estimates of $E_n(u_1)$ and $E_n(u_2)$ imply that $u_1$ and $u_2$ are not identically zero on $V_n$.
\end{myrmk}

\begin{proof}[Proof of Lemma \ref{lem_u1}]
Denote
$$S=\left(\myset{\frac{1}{3},\frac{2}{3}}\times\left[\frac{1}{3},\frac{2}{3}\right]\right)\bigcup\left(\left[\frac{1}{3},\frac{2}{3}\right]\times\myset{\frac{1}{3},\frac{2}{3}}\right).$$
For any $n\ge1$, let
$$R_n(S,L)=\inf\myset{E_n(u):u|_{S\cap V_n}=1,u|_{L\cap V_n}=0,u\in l(V_n)}^{-1},$$
then there exists a unique $u^{(n)}\in l(V_n)$ with $u^{(n)}|_{S\cap V_n}=1,u^{(n)}|_{L\cap V_n}=0$ such that $E_n(u^{(n)})^{-1}=R_n(S,L)$. By flow technique and potential technique from \cite{BB90,BBS90,McG02}, there exists a universal positive constant $C$ such that
$$\frac{1}{C}\rho^n\le R_n(S,L)\le C\rho^n\text{ for any }n\ge1.$$
By the same proof as \cite[Section 8]{GY19} using uniform Harnack inequality, we obtain $u_1\in C(K)$ satisfying $0\le u_1\le 1$ on $K$, $u_1|_S=1$, $u_1|_L=0$ and $a_n(u_1)\le cC$ for any $n\ge1$, where $c$ is the universal positive constant in weak monotonicity result \cite[Theorem 7.1]{GY19}. Hence
$$E_n(u_1)\le cC\left(\frac{1}{\rho}\right)^n\text{ for any }n\ge1.$$
By the optimal property of $u^{(n)}$, we have
$$E_n(u_1)\ge E_n(u^{(n)})=R_n(S,L)^{-1}\ge\frac{1}{C}\left(\frac{1}{\rho}\right)^n\text{ for any }n\ge1.$$
\end{proof}

\begin{proof}[Proof of Lemma \ref{lem_u2}]
Let $f\in C([0,1])$ be a strictly increasing function given by $f(0)=0$, $f(1)=1$ and for any $n\ge0$, for any $i=0,1,\ldots,3^{n}-1$
\begin{align*}
&f(\frac{3i+1}{3^{n+1}})=\frac{5}{7}f(\frac{i}{3^n})+\frac{2}{7}f(\frac{i+1}{3^n}),\\
&f(\frac{3i+2}{3^{n+1}})=\frac{2}{7}f(\frac{i}{3^n})+\frac{5}{7}f(\frac{i+1}{3^n}).
\end{align*}
By the proof on \cite[Page 4001]{GY19}, we have
$$E_n\left((x,y)\mapsto f(x)\right)=E_n\left((x,y)\mapsto f(y)\right)=\left(\frac{6}{7}\right)^n\text{ for any }n\ge1,$$
and for any $n\ge1$, $f$ is the unique optimal function of the following variational problem
$$\inf\myset{E_n\left((x,y)\mapsto f(x)\right)=E_n\left((x,y)\mapsto f(y)\right):f(0)=0,f(1)=1,f\in C([0,1])}.$$

Let $u_2\in C(K)$ be given by $u_2(x,y)=f(x)f(1-x)f(y)f(1-y)$, $(x,y)\in K$. It is obvious that $0\le u_2\le 1$ on $K$ and $u_2|_L=0$.

For any $n\ge1$, by Minkowski inequality, we have
$$\sqrt{E_n(uv)}\le\max_{V_n}|u|\sqrt{E_n(v)}+\max_{V_n}|v|\sqrt{E_n(u)}\text{ for any }u,v\in l(V_n),$$
hence
$$\sqrt{E_n(u_2)}\le4\sqrt{E_n\left((x,y)\mapsto f(x)\right)},$$
hence
$$E_n(u_2)\le16E_n\left((x,y)\mapsto f(x)\right)=16\left(\frac{6}{7}\right)^n\text{ for any }n\ge1.$$

For any $n\ge1$, we have
\begin{align*}
&E_{n+2}(u_2)\\
&\ge\sum_{w\in W_n}
{\sum_{\mbox{\tiny
$
\begin{subarray}{c}
p=(p_x,p_y),q=(q_x,q_y)\in V_{50w}\\
|p-q|=2^{-1}\cdot3^{-(n+2)}\\
p_y=q_y
\end{subarray}
$
}}}
\left(f(p_x)f(1-p_x)f(p_y)f(1-p_y)\right.\\
&\qquad\qquad\qquad\qquad\qquad\qquad\qquad\left.-f(q_x)f(1-q_x)f(q_y)f(1-q_y)\right)^2\\
&=\sum_{w\in W_n}
{\sum_{\mbox{\tiny
$
\begin{subarray}{c}
p=(p_x,p_y),q=(q_x,q_y)\in V_{50w}\\
|p-q|=2^{-1}\cdot3^{-(n+2)}\\
p_y=q_y
\end{subarray}
$
}}}
\left(f(p_x)f(1-p_x)-f(q_x)f(1-q_x)\right)^2f(p_y)^2f(1-p_y)^2.
\end{align*}
For any $p=(p_x,p_y)\in K_{50}$, we have $f(p_y)\in\left[\frac{5}{7},\frac{39}{49}\right]$, $f(1-p_y)\in\left[\frac{10}{49},\frac{2}{7}\right]$. Hence
\begin{align*}
&E_{n+2}(u_2)\\
&\ge\left(\frac{5}{7}\right)^2\left(\frac{10}{49}\right)^2
\sum_{w\in W_n}
{\sum_{\mbox{\tiny
$
\begin{subarray}{c}
p=(p_x,p_y),q=(q_x,q_y)\in V_{50w}\\
|p-q|=2^{-1}\cdot3^{-(n+2)}\\
p_y=q_y
\end{subarray}
$
}}}
\left(f(p_x)f(1-p_x)-f(q_x)f(1-q_x)\right)^2\\
&=\left(\frac{5}{7}\right)^2\left(\frac{10}{49}\right)^2
\sum_{w\in W_n}
{\sum_{\mbox{\tiny
$
\begin{subarray}{c}
p=(p_x,p_y),q=(q_x,q_y)\in V_{50w}\\
|p-q|=2^{-1}\cdot3^{-(n+2)}\\
\end{subarray}
$
}}}
\left(f(p_x)f(1-p_x)-f(q_x)f(1-q_x)\right)^2\\
&=\left(\frac{5}{7}\right)^2\left(\frac{10}{49}\right)^2E_n\left(\left((x,y)\mapsto f(x)f(1-x)\right)\circ f_{50}\right).
\end{align*}
Note that the function $\left((x,y)\mapsto f(x)f(1-x)\right)\circ f_{50}\in C(K)$ depends only on the first variable and maps $0$ to $\frac{10}{49}$ and $1$ to $\frac{580}{2401}$. By the optimal property of $f$, we have
$$E_n\left(\left((x,y)\mapsto f(x)f(1-x)\right)\circ f_{50}\right)\ge\left(\frac{580}{2401}-\frac{10}{49}\right)^2E_n\left((x,y)\mapsto f(x)\right)=\left(\frac{90}{2401}\right)^2\left(\frac{6}{7}\right)^n.$$
Hence
$$E_{n+2}(u_2)\ge\left(\frac{5}{7}\right)^2\left(\frac{10}{49}\right)^2\left(\frac{90}{2401}\right)^2\left(\frac{6}{7}\right)^n=\frac{562500}{13841287201}\left(\frac{6}{7}\right)^{n+2}\text{ for any }n\ge1.$$
Hence 
$$E_n(u_2)\ge C\left(\frac{6}{7}\right)^n\text{ for any }n\ge1,$$
where $C=\min\myset{\frac{7}{6}E_1(u_2),\left(\frac{7}{6}\right)^2E_2(u_2),\frac{562500}{13841287201}}$ is some universal positive constant.
\end{proof}

We only need to prove Proposition \ref{prop_SC} as follows.

\begin{proof}[Proof of Proposition \ref{prop_SC}]
In terms of $E_n$, we need to construct $u\in C(K)$ with a universal positive constant $C$ satisfying $E_n(u)\le Cn\left(\frac{1}{\rho}\right)^n$ for any $n\ge1$ and there exists some sequence $\myset{n_k}$ such that $E_{n_k}(u)\ge\frac{1}{C}n_k\left(\frac{1}{\rho}\right)^{n_k}$ for any $k\ge1$.

The idea of the proof is to use the fact that $\rho>\frac{7}{6}$ which implies that $\left(\frac{6}{7}\right)^n$ is much larger than $n\left(\frac{1}{\rho}\right)^n$ for any sufficiently large $n\ge1$ and the fact that $n\left(\frac{1}{\rho}\right)^n$ is much larger than $\left(\frac{1}{\rho}\right)^n$ for any sufficiently large $n\ge1$.

For any $A\subseteq K$, for any $n\ge1$, for any $u\in l(V_n)$, denote
$$E_n^A(u)=
{\sum_{\mbox{\tiny
$
\begin{subarray}{c}
w\in W_n\\
K_w\subseteq\mybar{A}
\end{subarray}
$
}}}
{\sum_{\mbox{\tiny
$
\begin{subarray}{c}
p,q\in V_w\\
|p-q|=2^{-1}\cdot3^{-n}
\end{subarray}
$
}}}
(u(p)-u(q))^2.$$

Let $C_1$ be the universal positive constant from Lemma \ref{lem_u1}. For simplicity, we may assume that $C_1$ is an integer. Take arbitrary integer $C\ge C_1+1$.

Let $u^{(1)}=\delta_2u_2$, where $\delta_2$ is some positive parameter to be determined, then
$$E_n(u^{(1)})=\delta_2^2E_{n}(u_2)\text{ for any }n\ge1.$$
Take $\delta_2>0$ such that
$$E_n(u^{(1)})=\delta_2^2E_n(u_2)<Cn\left(\frac{1}{\rho}\right)^n\text{ for any }n=1,2.$$
By Lemma \ref{lem_u2}, there exists $n_1>2$ such that
\begin{align*}
E_n(u^{(1)})&<Cn\left(\frac{1}{\rho}\right)^n\text{ for any }n=1,\ldots,n_1-1,\\
E_{n_1}(u^{(1)})&\ge Cn_1\left(\frac{1}{\rho}\right)^{n_1}.
\end{align*}

Define $u^{(2)}\in l(K)$ as follows. For any $w\in W_{n_1-1}$, let
$$
u^{(2)}|_{K_{w}}=
\begin{cases}
\delta_1 u_1\circ f_w,&\text{if }w\ne0^{n_1-1},\\
\delta_2 u_2\circ f_w,&\text{if }w=0^{n_1-1},
\end{cases}
$$
where $\delta_1,\delta_2$ are some positive parameters to be determined. Since $u_1|_L=u_2|_L=0$, we have $u^{(2)}\in C(K)$ is well-defined and
$$E_n(u^{(2)})=
\begin{cases}
0,&\text{if }n=1,\ldots,n_1-1,\\
\delta_1^2(8^{n_1-1}-1)E_{n-n_1+1}(u_1)+\delta_2^2E_{n-n_1+1}(u_2),&\text{if }n=n_1,n_1+1,\ldots.
\end{cases}
$$
Take $\delta_1\in(0,+\infty)$ such that
$$\delta_1^2(8^{n_1-1}-1)E_{1}(u_1)=\frac{1}{2C}n_1\left(\frac{1}{\rho}\right)^{n_1},$$
then
$$\delta_1=\sqrt{\frac{\frac{1}{2C}n_1\left(\frac{1}{\rho}\right)^{n_1}}{(8^{n_1-1}-1)E_{1}(u_1)}}\le\sqrt{\frac{n_1\left(\frac{1}{\rho}\right)^{n_1}}{14CE_{1}(u_1)}}.$$
By Lemma \ref{lem_u1}, we have
\begin{align*}
&\delta_1^2(8^{n_1-1}-1)E_{n-n_1+1}(u_1)\le C_1^2\left(\frac{1}{\rho}\right)^{n-n_1}\delta_1^2(8^{n_1-1}-1)E_{1}(u_1)\\
&=\frac{C_1^2}{2C}n_1\left(\frac{1}{\rho}\right)^{n}\le\frac{C_1^2}{2C}n\left(\frac{1}{\rho}\right)^n<Cn\left(\frac{1}{\rho}\right)^n\text{ for any }n=n_1,n_1+1,\ldots.
\end{align*}
Take $\delta_2\in(0,1)$ such that
\begin{align*}
E_{n}(u^{(2)})&=\delta_1^2(8^{n_1-1}-1)E_{n-n_1+1}(u_1)+\delta_2^2E_{n-n_1+1}(u_2)<Cn\left(\frac{1}{\rho}\right)^{n}\\
&\text{ for any }n=n_1,n_1+1,\ldots,2C_1^2C^2n_1,2C_1^2C^2n_1+1.
\end{align*}
Hence
$$E_n(u^{(2)})<Cn\left(\frac{1}{\rho}\right)^n\text{ for any }n=1,\ldots,2C_1^2C^2n_1,2C_1^2C^2n_1+1,$$
$$E_{n_1}^{K\backslash K_{0^{n_1-1}}}(u^{(2)})\ge\frac{1}{2C}n_1\left(\frac{1}{\rho}\right)^{n_1}.$$

Assume that we have constructed $u^{(k)}\in C(K)$ and $n_{k-1}>\ldots>n_1>2$ satisfying that $u^{(k)}|_{K_{0^{n_{k-1}-1}}}$ is the product of some positive parameter and $u_2\circ f_{0^{n_{k-1}-1}}$, $u^{(k)}|_{K\backslash K_{0^{n_{k-1}-1}}}$ is constructed by gluing the terms of the form $\delta u_1\circ f_w$ with $\delta\in(0,+\infty)$ and $w\in W_1\cup\ldots\cup W_{n_{k-1}-1}$, and
$$E_n(u^{(k)})<Cn\left(\frac{1}{\rho}\right)^n\text{ for any }n=1,\ldots,2C_1^2C^2n_{k-1},2C_1^2C^2n_{k-1}+1,$$
$$E^{K\backslash K_{0^{n_{k-1}-1}}}_{n_i}(u^{(k)})\ge\frac{1}{2C}n_i\left(\frac{1}{\rho}\right)^{n_i}\text{ for any }i=1,\ldots,k-1.$$
By Lemma \ref{lem_u2}, there exists $n_k>2C_1^2C^2n_{k-1}+1$ such that
\begin{align*}
E_n(u^{(k)})&<Cn\left(\frac{1}{\rho}\right)^n\text{ for any }n=1,\ldots, n_{k}-1,\\
E_{n_k}(u^{(k)})&\ge C{n_k}\left(\frac{1}{\rho}\right)^{n_k}.
\end{align*}

Define $u^{(k+1)}\in l(K)$ as follows. Let $u^{(k+1)}|_{K\backslash K_{0^{n_{k-1}-1}}}=u^{(k)}|_{K\backslash K_{0^{n_{k-1}-1}}}$ and for any $w\in W_{n_k-n_{k-1}}$, let
$$u^{(k+1)}|_{K_{0^{n_{k-1}-1}w}}=
\begin{cases}
\delta_1u_1\circ f_{0^{n_{k-1}-1}w},&\text{if }w\ne0^{n_{k}-n_{k-1}},\\
\delta_2u_2\circ f_{0^{n_{k-1}-1}w},&\text{if }w=0^{n_{k}-n_{k-1}},
\end{cases}
$$
where $\delta_1,\delta_2$ are some positive parameters to be determined, see Figure \ref{fig_ukplus1}.

\begin{figure}[ht]
\centering
\begin{tikzpicture}[scale=1/3]

\foreach \x in {0,1,2,...,27}
\draw (\x,0)--(\x,27);

\foreach \y in {0,1,2,...,27}
\draw (0,\y)--(27,\y);

\foreach \x in {3,12,21}
\foreach \y in {3,12,21}
\draw[fill=white] (\x,\y)--(\x+3,\y)--(\x+3,\y+3)--(\x,\y+3)--cycle;

\draw[fill=white] (9,9)--(18,9)--(18,18)--(9,18)--cycle;

\draw [decorate,decoration={brace,amplitude=2pt},xshift=-2,yshift=0] (0,0)--(0,1) node [black,midway,xshift=-15] {{\tiny{$K_{0^{n_k-1}}$}}};

\draw [decorate,decoration={brace,amplitude=10pt,mirror},xshift=2] (27,0)--(27,27) node [black,midway,xshift=33,yshift=-2] {{\footnotesize{$K_{0^{n_{k-1}-1}}$}}};

\end{tikzpicture}
\caption{The Construction of $u^{(k+1)}$}\label{fig_ukplus1}
\end{figure}
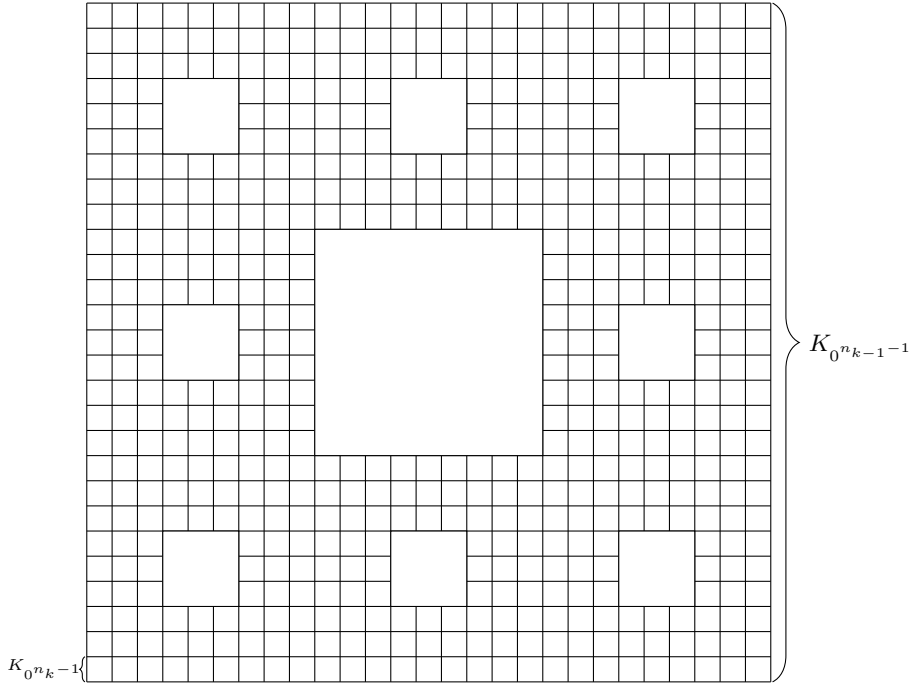

Since $u_1|_L=u_2|_L=0$, we have $u^{(k+1)}\in C(K)$ is well-defined, $u^{(k+1)}|_{K_{0^{n_{k}-1}}}$ is the product of some positive parameter and $u_2\circ f_{0^{n_{k}-1}}$, $u^{(k+1)}|_{K\backslash K_{0^{n_{k}-1}}}$ is constructed by gluing the terms of the form $\delta u_1\circ f_w$ with $\delta\in(0,+\infty)$ and $w\in W_1\cup\ldots\cup W_{n_{k}-1}$, and
\begin{align*}
&E_n(u^{(k+1)})\\
&=
\begin{cases}
E_n^{K\backslash K_{0^{n_{k-1}-1}}}(u^{(k)}),&\text{if }n=1,\ldots,n_k-1,\\
E_n^{K\backslash K_{0^{n_{k-1}-1}}}(u^{(k)})+\delta_1^2(8^{n_k-n_{k-1}}-1)E_{n-n_k+1}(u_1)\\
\qquad\qquad\qquad\qquad\qquad+\delta_2^2E_{n-n_k+1}(u_2),&\text{if }n=n_k,n_k+1,\ldots.\\
\end{cases}
\end{align*}
Hence
$$E_n(u^{(k+1)})=E_n^{K\backslash K_{0^{n_{k-1}-1}}}(u^{(k)})\le E_n(u^{(k)})<Cn\left(\frac{1}{\rho}\right)^n\text{ for any }n=1,\ldots,n_k-1,$$
\begin{align*}
E^{K\backslash K_{0^{n_k-1}}}_{n_i}(u^{(k+1)})\ge E^{K\backslash K_{0^{n_{k-1}-1}}}_{n_i}(u^{(k+1)})&=E^{K\backslash K_{0^{n_{k-1}-1}}}_{n_i}(u^{(k)})\ge\frac{1}{2C}n_i\left(\frac{1}{\rho}\right)^{n_i}\\
&\text{ for any }i=1,\ldots,k-1.
\end{align*}

Since $u^{(k+1)}|_{K\backslash K_{0^{n_{k-1}-1}}}=u^{(k)}|_{K\backslash K_{0^{n_{k-1}-1}}}$ is constructed by gluing the terms of the form $\delta u_1\circ f_w$ with $\delta\in(0,+\infty)$ and $w\in W_1\cup\ldots\cup W_{n_{k-1}-1}$, by Lemma \ref{lem_u1}, we have
\begin{align*}
&E_n^{K\backslash K_{0^{n_{k-1}-1}}}(u^{(k)})\le C_1^2\left(\frac{1}{\rho}\right)^{n-n_{k-1}}E_{n_{k-1}}^{K\backslash K_{0^{n_{k-1}-1}}}(u^{(k)})\le C_1^2\left(\frac{1}{\rho}\right)^{n-n_{k-1}}E_{n_{k-1}}(u^{(k)})\\
&< C_1^2\left(\frac{1}{\rho}\right)^{n-n_{k-1}}Cn_{k-1}\left(\frac{1}{\rho}\right)^{n_{k-1}}=C_1^2Cn_{k-1}\left(\frac{1}{\rho}\right)^n\text{ for any }n=n_{k-1},n_{k-1}+1,\ldots.
\end{align*}
Hence
\begin{align*}
E_n^{K\backslash K_{0^{n_{k-1}-1}}}(u^{(k)})&<C_1^2Cn_{k-1}\left(\frac{1}{\rho}\right)^n<\frac{1}{2C}n\left(\frac{1}{\rho}\right)^n\\
&\text{ for any }n=2C_1^2C^2n_{k-1}+1,2C_1^2C^2n_{k-1}+2,\ldots.
\end{align*}
In particular
$$E_{n_k}^{K\backslash K_{0^{n_{k-1}-1}}}(u^{(k)})<\frac{1}{2C}n_k\left(\frac{1}{\rho}\right)^{n_k}.$$
Take $\delta_1\in(0,+\infty)$ such that
$$E_{n_k}^{K\backslash K_{0^{n_k-1}}}(u^{(k+1)})=E_{n_{k}}^{K\backslash K_{0^{n_{k-1}-1}}}(u^{(k)})+\delta_1^2(8^{n_k-n_{k-1}}-1)E_1(u_1)=\frac{1}{2C}n_k\left(\frac{1}{\rho}\right)^{n_k},$$
then
$$\delta_1\le\sqrt{\frac{\frac{1}{2C}n_k\left(\frac{1}{\rho}\right)^{n_k}}{(8^{n_k-n_{k-1}}-1)E_1(u_1)}}\le\sqrt{\frac{n_k\left(\frac{1}{\rho}\right)^{n_k}}{14CE_1(u_1)}}.$$
Since $u^{(k+1)}|_{K\backslash K_{0^{n_{k-1}-1}}}=u^{(k)}|_{K\backslash K_{0^{n_{k-1}-1}}}$ is constructed by gluing the terms of the form $\delta u_1\circ f_w$ with $\delta\in(0,+\infty)$ and $w\in W_1\cup\ldots\cup W_{n_{k-1}-1}$, by Lemma \ref{lem_u1}, we have
\begin{align*}
&E_{n}^{K\backslash K_{0^{n_{k-1}-1}}}(u^{(k)})+\delta_1^2(8^{n_k-n_{k-1}}-1)E_{n-n_k+1}(u_1)\\
&\le C_1^2\left(\frac{1}{\rho}\right)^{n-n_k}\left(E_{n_k}^{K\backslash K_{0^{n_{k-1}-1}}}(u^{(k)})+\delta_1^2(8^{n_k-n_{k-1}}-1)E_1(u_1)\right)\\
&=\frac{C_1^2}{2C}n_k\left(\frac{1}{\rho}\right)^{n}\le\frac{C_1^2}{2C}n\left(\frac{1}{\rho}\right)^{n}<Cn\left(\frac{1}{\rho}\right)^{n}\text{ for any }n=n_k,n_k+1,\ldots.
\end{align*}
Take $\delta_2\in(0,\frac{1}{k^2})$ such that
\begin{align*}
E_n(u^{(k+1)})&=E_{n}^{K\backslash K_{0^{n_{k-1}-1}}}(u^{(k)})+\delta_1^2(8^{n_k-n_{k-1}}-1)E_{n-n_k+1}(u_1)+\delta_2^2E_{n-n_k+1}(u_2)\\
&<Cn\left(\frac{1}{\rho}\right)^n\text{ for any }n=n_k,n_k+1,\ldots,2C_1^2C^2n_k,2C_1^2C^2n_k+1.
\end{align*}

Therefore, we have
$$E_n(u^{(k+1)})<Cn\left(\frac{1}{\rho}\right)^n\text{ for any }n=1,\ldots,2C_1^2C^2n_k,2C_1^2C^2n_k+1,$$
$$E^{K\backslash K_{0^{n_k-1}}}_{n_i}(u^{(k+1)})\ge\frac{1}{2C}n_i\left(\frac{1}{\rho}\right)^{n_i}\text{ for any }i=1,\ldots,k.$$

By induction principle, we obtain $\myset{u^{(k)}}\subseteq C(K)$.

By construction, for any $k\ge2$, we have
$$u^{(k)}|_{K\backslash K_{0^{n_{k-1}-1}}}=u^{(k+1)}|_{K\backslash K_{0^{n_{k-1}-1}}},$$
$$0\le u^{(k)}|_{K_{0^{n_{k-1}-1}}}\le\frac{1}{(k-1)^2},$$
$$0\le u^{(k+1)}|_{K_{0^{n_{k-1}-1}}}\le\sqrt{\frac{n_k\left(\frac{1}{\rho}\right)^{n_k}}{14CE_1(u_1)}}+\frac{1}{k^2},$$
hence
$$\sup_{K}|u^{(k)}-u^{(k+1)}|\le\sqrt{\frac{n_k\left(\frac{1}{\rho}\right)^{n_k}}{14CE_1(u_1)}}+\frac{1}{k^2}+\frac{1}{(k-1)^2}.$$
Hence for any $k>l\ge2$, we have
$$\sup_K|u^{(k)}-u^{(l)}|\le\sum_{i=l}^{k-1}\sup_K|u^{(i)}-u^{(i+1)}|\le\sum_{i=l}^{k-1}\left(\sqrt{\frac{n_i\left(\frac{1}{\rho}\right)^{n_i}}{14CE_1(u_1)}}+\frac{1}{i^2}+\frac{1}{(i-1)^2}\right)\to0$$
as $k,l\to+\infty$. Hence $\myset{u^{(k)}}\subseteq C(K)$ converges uniformly to some function $u\in C(K)$.

For any fixed $n\ge1$, there exists $l\ge1$ such that $n\le 2C_1^2C^2n_{k}+1$ for any $k\ge l$, hence
$$E_n(u^{(k+1)})<Cn\left(\frac{1}{\rho}\right)^n\text{ for any }k\ge l.$$
Letting $k\to+\infty$, we have
$$E_n(u)=\lim_{k\to+\infty}E_n(u^{(k+1)})\le Cn\left(\frac{1}{\rho}\right)^n\text{ for any }n\ge1.$$

For any fixed $i\ge1$, for any $k\ge i$, we have
$$E_{n_i}(u^{(k+1)})\ge E_{n_i}^{K\backslash K_{0^{n_k-1}}}(u^{(k+1)})\ge\frac{1}{2C}n_i\left(\frac{1}{\rho}\right)^{n_i}.$$
Letting $k\to+\infty$, we have
$$E_{n_i}(u)=\lim_{k\to+\infty}E_{n_i}(u^{(k+1)})\ge\frac{1}{2C}n_i\left(\frac{1}{\rho}\right)^{n_i}\text{ for any }i\ge1.$$

Therefore, $u\in C(K)$ is our desired function.
\end{proof}

\section{Proof of Theorem \ref{thm_subordinate}}\label{sec_subordinate}

First, we list some results about closed form.

Let $(\calE,\calF)$ on $L^2(K;\nu)$ be a closed form that corresponds to a strongly continuous semi-group $\myset{T_t:t>0}$ on $L^2(K;\nu)$. By \cite[Section 1.3]{FOT11}, there exists a spectral family $\myset{E_\lambda:\lambda\in[0,+\infty)}$ such that
\begin{align*}
&\calE(u,u)=\int_{[0,+\infty)}\lambda\md(E_\lambda u,u),\\
&\calF=\myset{u\in L^2(K;\nu):\int_{[0,+\infty)}\lambda\md(E_\lambda u,u)<+\infty},
\end{align*}
and
$$T_t=\int_{[0,+\infty)}e^{-t\lambda}\md E_\lambda\text{ for any }t\in(0,+\infty).$$

For any $t\in(0,+\infty)$, for any $u\in L^2(K;\nu)$, let
$$\calE_{(t)}(u,u)=\frac{1}{t}(u-T_tu,u).$$

We have the following result.

\begin{mylem}\label{lem_Et}(\cite[Lemma 1.3.4]{FOT11})
For any $u\in L^2(K;\nu)$, we have $t\mapsto\calE_{(t)}(u,u)$ is monotone decreasing in $(0,+\infty)$ and
\begin{align*}
&\calE(u,u)=\lim_{t\downarrow0}\calE_{(t)}(u,u),\\
&\calF=\myset{u\in L^2(K;\nu):\lim_{t\downarrow0}\calE_{(t)}(u,u)<+\infty}.
\end{align*}
\end{mylem}

Second, we give some results about subordinated Hunt process.

Let $(\calE,\calF)$ be a regular Dirichlet form on $L^2(K;\nu)$ that corresponds to a Hunt process $\myset{X_t}$. For any $\delta\in(0,1)$, let $\myset{\xi^{(\delta)}_t}$ be the $\delta$-stable subordinator, that is, the one-dimensional L\'evy process whose Laplace transform is given by $\bbE e^{-s\xi^{(\delta)}_t}=e^{-ts^{\delta}}$, let $\eta^{(\delta)}_t$ be its one-dimensional distribution density. Assume that the processes $\myset{X_t}$ and $\myset{\xi_t^{(\delta)}}$ are independent, then the $\delta$-subordinated Hunt process$\myset{X^{(\delta)}_t}$ is given by $\myset{X_{\xi_t^{(\delta)}}}$. Let $\myset{T^{(\delta)}_t:t>0}$ on $L^2(K;\nu)$ be the strongly continuous Markovian semi-group and $(\calE^{(\delta)},\calF^{(\delta)})$ on $L^2(K;\nu)$ the regular Dirichlet form that correspond to the Hunt process $\myset{X^{(\delta)}_t}$, then
\begin{equation}\label{eqn_sub}
T^{(\delta)}_tu=\int_0^{+\infty}T_su\left(\eta^{(\delta)}_t(s)\md s\right)\text{ for any }t\in(0,+\infty),\text{ for any }u\in L^2(K;\nu),
\end{equation}
see \cite{Ber96,Sat99}.

We have the characterization of $\calE^{(\delta)}$ as follows.

\begin{myprop}\label{prop_pre}
For any $\delta\in(0,1)$, for any $u\in L^2(K;\nu)$, we have
$$\calE^{(\delta)}(u,u)=\int_{[0,+\infty)}\lambda^\delta\md(E_\lambda u,u)=\frac{\delta}{\Gamma(1-\delta)}\int_0^{+\infty}\frac{1}{s^{\delta}}\calE_{(s)}(u,u)\md s\left(\le+\infty\right).$$
\end{myprop}

\begin{myrmks}
\begin{enumerate}[(1)]
\item For any $u\in L^2(K;\nu)$. $t^{-1}(u-T_t^{(\delta)}u,u)$ is non-negative finite for any $t\in(0,+\infty)$, monotone decreasing in $t\in(0,+\infty)$ and $\calE^{(\delta)}(u,u)$ is defined as its limit as $t\downarrow0$ by Lemma \ref{lem_Et} which is allowed to be $+\infty$. $\md(E_\lambda u,u)$ is an ordinary measure on $[0,+\infty)$ and $\calE_{(s)}(u,u)$ is non-negative finite for any $s\in(0,+\infty)$, hence the above two integrals are well-defined and allowed also to be $+\infty$.
\item \cite[Equation (3.5)]{Pie08} gave the above second equality only for any $u\in\calF$ where the condition $u\in\calF$ is intrinsically used to apply dominated convergence theorem.
\end{enumerate}
\end{myrmks}

\begin{proof}
For any $u\in L^2(K;\nu)$, by Lemma \ref{lem_Et} and Equation (\ref{eqn_sub}), we have
\begin{align*}
&\calE^{(\delta)}(u,u)=\lim_{t\downarrow0}\frac{1}{t}\left(u-T_t^{(\delta)}u,u\right)\\
&=\lim_{t\downarrow0}\frac{1}{t}\left((u,u)-\int_0^{+\infty}(T_su,u)\eta^{(\delta)}_t(s)\md s\right)\\
&=\lim_{t\downarrow0}\frac{1}{t}\int_0^{+\infty}\left(u-T_su,u\right)\eta_t^{(\delta)}(s)\md s\\
&=\lim_{t\downarrow0}\frac{1}{t}\int_0^{+\infty}\left(\int_{[0,+\infty)}(1-e^{-s\lambda})\md(E_\lambda u,u)\right)\eta_t^{(\delta)}(s)\md s\\
&=\lim_{t\downarrow0}\frac{1}{t}\int_{[0,+\infty)}\left(\int_0^{+\infty}(1-e^{-s\lambda})\eta_t^{(\delta)}(s)\md s\right)\md(E_\lambda u,u)\\
&=\lim_{t\downarrow0}\frac{1}{t}\int_{[0,+\infty)}\left(1-e^{-t\lambda^\delta}\right)\md(E_\lambda u,u).
\end{align*}
Since $t\mapsto\frac{1-e^{-t\lambda^\delta}}{t}$ is monotone decreasing in $(0,+\infty)$ for any $\lambda\in[0,+\infty)$, by monotone convergence theorem, we have
\begin{align*}
&\calE^{(\delta)}(u,u)=\lim_{t\downarrow0}\frac{1}{t}\int_{[0,+\infty)}\left(1-e^{-t\lambda^\delta}\right)\md(E_\lambda u,u)\\
&=\int_{[0,+\infty)}\lim_{t\downarrow0}\frac{1-e^{-t\lambda^\delta}}{t}\md(E_\lambda u,u)=\int_{[0,+\infty)}\lambda^\delta\md(E_\lambda u,u).
\end{align*}

Recall that for any $\delta\in(0,1)$, we have
$$\int_0^{+\infty}\frac{1-e^{-s}}{s^{1+\delta}}\md s=\frac{\Gamma(1-\delta)}{\delta}$$
which implies that for any $\lambda\in[0,+\infty)$, we have
$$\int_0^{+\infty}\frac{1-e^{-s\lambda}}{s^{1+\delta}}\md s=\frac{\Gamma(1-\delta)}{\delta}\lambda^{\delta}.$$

Hence
\begin{align*}
&\calE^{(\delta)}(u,u)=\int_{[0,+\infty)}\lambda^\delta\md(E_\lambda u,u)=\frac{\delta}{\Gamma(1-\delta)}\int_{[0,+\infty)}\left(\int_0^{+\infty}\frac{1-e^{-s\lambda}}{s^{1+\delta}}\md s\right)\md(E_\lambda u,u)\\
&=\frac{\delta}{\Gamma(1-\delta)}\int_0^{+\infty}\frac{1}{s^{1+\delta}}\left(\int_{[0,+\infty)}(1-e^{-s\lambda})\md(E_\lambda u,u)\right)\md s\\
&=\frac{\delta}{\Gamma(1-\delta)}\int_0^{+\infty}\frac{1}{s^{1+\delta}}(u-T_su,u)\md s=\frac{\delta}{\Gamma(1-\delta)}\int_0^{+\infty}\frac{1}{s^{\delta}}\calE_{(s)}(u,u)\md s.
\end{align*}
\end{proof}

We have some direct corollaries as follows.

\begin{mycor}\label{cor_domain}\hspace{0em}
\begin{enumerate}[(1)]
\item For any $\delta\in(0,1)$, we have
\begin{align*}
\calF^{(\delta)}&=\myset{u\in L^2(K;\nu):\int_{[0,+\infty)}\lambda^\delta\md(E_\lambda u,u)<+\infty}\\
&=\myset{u\in L^2(K;\nu):\int_{(1,+\infty)}\lambda^\delta\md(E_\lambda u,u)<+\infty}.
\end{align*}
\item For any $\delta_1,\delta_2\in(0,1)$ with $\delta_1<\delta_2$, we have $\calF^{(\delta_1)}\supseteq\calF^{(\delta_2)}\supseteq\calF$.
\end{enumerate}
\end{mycor}

\begin{proof}
(1) The first equality follows directly from Proposition \ref{prop_pre}. Since for any $u\in L^2(K;\nu)$, we have
$$\int_{[0,1]}\lambda^{\delta}\md(E_\lambda u,u)\le\int_{[0,1]}\md(E_\lambda u,u)\le\int_{[0,+\infty)}\md(E_\lambda u,u)=(u,u)<+\infty.$$
Hence
$$\int_{[0,+\infty)}\lambda^\delta\md(E_\lambda u,u)<+\infty$$
if and only if
$$\int_{(1,+\infty)}\lambda^\delta\md(E_\lambda u,u)<+\infty.$$
Hence we have the second equality.

(2) It follows easily from (1).
\end{proof}

\begin{mycor}\label{cor_regular}
For any $\delta\in(0,1)$.
\begin{enumerate}[(1)]
\item $\calF$ is $(\calE^{(\delta)}(\cdot,\cdot)+(\cdot,\cdot))$-dense in $\calF^{(\delta)}$.
\item Any core of $(\calE,\calF)$ on $L^2(K;\nu)$ is a core of $(\calE^{(\delta)},\calF^{(\delta)})$ on $L^2(K;\nu)$.
\end{enumerate}
\end{mycor}

\begin{myrmk}
If $(\calE^{(\delta)},\calF^{(\delta)})$ on $L^2(K;\nu)$ is defined only as the Dirichlet form corresponding to the strongly continuous Markovian semi-group $\myset{T^{(\delta)}_t:t>0}$ on $L^2(K;\nu)$ which is given by Equation (\ref{eqn_sub}), then the regular property of $(\calE^{(\delta)},\calF^{(\delta)})$ on $L^2(K;\nu)$ follows also from the regular property of $(\calE,\calF)$ on $L^2(K;\nu)$ and the above result.
\end{myrmk}

\begin{proof}
(1) For any $t\in(0,+\infty)$, for any $u\in L^2(K;\nu)$, we claim that $T_t^{(\delta)}u\in\calF$. We only need to show that
$$\int_{[0,+\infty)}\lambda\md(E_\lambda T_t^{(\delta)}u,T_t^{(\delta)}u)<+\infty.$$
Indeed
\begin{align*}
&\int_{[0,+\infty)}\lambda\md(E_\lambda T_t^{(\delta)}u,T_t^{(\delta)}u)=\int_{[0,+\infty)}\int_0^{+\infty}\int_0^{+\infty}\lambda\eta^{(\delta)}_t(r)\eta^{(\delta)}_t(s)\md(E_\lambda T_ru,T_su)\md r\md s\\
&=\int_{[0,+\infty)}\lambda\left(\int_0^{+\infty}\int_0^{+\infty}e^{-r\lambda}e^{-s\lambda}\eta^{(\delta)}_t(r)\eta_t^{(\delta)}(s)\md r\md s\right)\md(E_\lambda u,u)\\
&=\int_{[0,+\infty)}\lambda e^{-2t\lambda^\delta}\md(E_\lambda u,u).
\end{align*}
Since $\lambda\mapsto\lambda e^{-2t\lambda^\delta}$ is continuous on $[0,+\infty)$ and $\lim_{\lambda\to+\infty}\lambda e^{-2t\lambda^\delta}=0$, there exists some positive constant $C$ such that
$$0\le\lambda e^{-2t\lambda^\delta}\le C\text{ for any }\lambda\in[0,+\infty).$$
Hence
$$\int_{[0,+\infty)}\lambda\md(E_\lambda T_t^{(\delta)}u,T_t^{(\delta)}u)=\int_{[0,+\infty)}\lambda e^{-2t\lambda^\delta}\md(E_\lambda u,u)\le C(u,u)<+\infty.$$

For any $u\in\calF^{(\delta)}$, by \cite[Lemma 1.3.3 (\rmnum{3})]{FOT11}, we have $T_t^{(\delta)}u\in\calF$ is $(\calE^{(\delta)}(\cdot,\cdot)+(\cdot,\cdot))$-convergent to $u$ as $t\downarrow0$. Hence $\calF$ is $(\calE^{(\delta)}(\cdot,\cdot)+(\cdot,\cdot))$-dense in $\calF^{(\delta)}$.

(2) For any $u\in\calF$, by Proposition \ref{prop_pre}, we have
\begin{align*}
\calE^{(\delta)}(u,u)&=\int_{[0,1]}\lambda^\delta\md(E_\lambda u,u)+\int_{(1,+\infty)}\lambda^\delta\md(E_\lambda u,u)\\
&\le\int_{[0,1]}\md(E_\lambda u,u)+\int_{(1,+\infty)}\lambda\md(E_\lambda u,u)\\
&\le\int_{[0,+\infty)}\md(E_\lambda u,u)+\int_{[0,+\infty)}\lambda\md(E_\lambda u,u)\\
&=\calE(u,u)+(u,u).
\end{align*}
Hence
$$\calE^{(\delta)}(u,u)+(u,u)\le 2\left(\calE(u,u)+(u,u)\right)\text{ for any }u\in\calF.$$

Let $\calC$ be a core of $(\calE,\calF)$ on $L^2(K;\nu)$, that is, $\calC$ is $(\calE(\cdot,\cdot)+(\cdot,\cdot))$-dense in $\calF$ and uniformly dense in $C_c(K)$. We only need to show that $\calC$ is $(\calE^{(\delta)}(\cdot,\cdot)+(\cdot,\cdot))$-dense in $\calF^{(\delta)}$. Indeed, by the above inequality, we have $\calC$ is $(\calE^{(\delta)}(\cdot,\cdot)+(\cdot,\cdot))$-dense in $\calF$. Since $\calF$ is $(\calE^{(\delta)}(\cdot,\cdot)+(\cdot,\cdot))$-dense in $\calF^{(\delta)}$ which is (1), we have $\calC$ is $(\calE^{(\delta)}(\cdot,\cdot)+(\cdot,\cdot))$-dense in $\calF^{(\delta)}$.
\end{proof}

\begin{mycor}\label{cor_BD}
Let $p_t(x,\md y)$ be the transition density of the regular Dirichlet form $(\calE,\calF)$ on $L^2(K;\nu)$. Then for any $\delta\in(0,1)$, we have
$$\calE^{(\delta)}(u,u)=\frac{1}{2}\int_K\int_K(u(x)-u(y))^2J^{(\delta)}(\md x\md y)+\int_Ku(x)^2k^{(\delta)}(\md x)\text{ for any }u\in L^2(K;\nu),$$
where
\begin{align*}
J^{(\delta)}(\md x\md y)&=\frac{\delta}{\Gamma(1-\delta)}\int_0^{+\infty}\frac{1}{s^{1+\delta}}p_s(x,\md y)\nu(\md x)\md s,\\
k^{(\delta)}(\md x)&=\frac{\delta}{\Gamma(1-\delta)}\int_0^{+\infty}\frac{1}{s^{1+\delta}}\left(1-\int_Kp_s(x,\md y)\right)\nu(\md x)\md s.
\end{align*}
\end{mycor}

\begin{myrmk}
The above result is indeed the Beurling-Deny decomposition of the regular Dirichlet form $(\calE^{(\delta)},\calF^{(\delta)})$ on $L^2(K;\nu)$ which has only jumping part and killing part, see \cite[Theorem 3.2.1, Lemma 4.5.4]{FOT11}. Hence $(\calE^{(\delta)},\calF^{(\delta)})$ on $L^2(K;\nu)$ is always non-local.
\end{myrmk}

\begin{proof}
For any $u\in L^2(K;\nu)$, for any $s\in(0,+\infty)$, we have
$$(u-T_su,u)=\frac{1}{2}\int_K\int_K(u(x)-u(y))^2p_s(x,\md y)\nu(\md x)+\int_Ku(x)^2\left(1-\int_Kp_s(x,\md y)\right)\nu(\md x).$$
Then the result follows directly from Proposition \ref{prop_pre}.
\end{proof}

Third, we give the proof of Theorem \ref{thm_subordinate} as follows.

\begin{proof}[Proof of Theorem \ref{thm_subordinate}]
It follows directly from Corollary \ref{cor_domain} (2) that $\calF\subseteq\cap_{\delta\in(0,1)}\calF^{(\delta)}$. If $A$ is bounded, then 
$$L^2(K;\nu)=\calD(A)\subseteq\calD(\sqrt{-A})=\calF\subseteq\bigcap_{\delta\in(0,1)}\calF^{(\delta)}\subseteq L^2(K;\nu),$$
hence we have ``$=$". We only need to show that if $A$ is unbounded, then ``$\subsetneqq$" holds.

Let $E$ be the spectral measure corresponding to the spectral family $\myset{E_\lambda:\lambda\in[0,+\infty)}$, that is,
$$E((\lambda_1,\lambda_2])=E_{\lambda_2}-E_{\lambda_1}\text{ for any }\lambda_1,\lambda_2\in[0,+\infty)\text{ with }\lambda_1<\lambda_2.$$
Then $\myset{E((2^k,2^{k+1}])}_{k\ge0}$ is a sequence of orthogonal projections on $L^2(K;\nu)$ satisfying
$$E((2^k,2^{k+1}])E((2^l,2^{l+1}])=0\text{ for any }k,l\ge0\text{ with }k\ne l.$$
Let
$$I=\myset{k\ge0:E((2^k,2^{k+1}])\ne0}.$$
Since $A$ is unbounded, the spectrum $\sigma(-A)\subseteq[0,+\infty)$ is unbounded, $\#I=+\infty$.

For any $k\in I$, there exists $u_k\in E((2^k,2^{k+1}])(L^2(K;\nu))$ with $\mynorm{u_k}^2_{L^2(K;\nu)}=2^{-(k+1)}$. For any $k\not\in I$, let $u_k=0$. Then $(u_k,u_l)=0$ for any $k,l\ge0$ with $k\ne l$.

Let $u=\sum_{k=0}^\infty u_k$. Then $u\in L^2(K;\nu)$ and
$$\mynorm{u}_{L^2(K;\nu)}^2=\sum_{k=0}^{\infty}\mynorm{u_k}_{L^2(K;\nu)}^2\le\sum_{k=0}^{\infty}\frac{1}{2^{k+1}}=1.$$

For any $k\ge0$, we have $E((2^k,2^{k+1}])u=u_k$. Hence
\begin{align*}
&\int_{(1,+\infty)}\lambda\md(E_\lambda u,u)=\sum_{k=0}^\infty\int_{(2^k,2^{k+1}]}\lambda\md(E_\lambda u,u)\ge\sum_{k\in I}\int_{(2^k,2^{k+1}]}\lambda\md(E_\lambda u,u)\\
&\ge\sum_{k\in I}\int_{(2^k,2^{k+1}]}2^k\md(E_\lambda u,u)=\sum_{k\in I}2^k(E((2^k,2^{k+1}])u,u)=\sum_{k\in I}2^k(u_k,u)\\
&=\sum_{k\in I}2^k\mynorm{u_k}_{L^2(K;\nu)}^2=\sum_{k\in I}2^{k}\frac{1}{2^{k+1}}=\frac{1}{2}\#I=+\infty,
\end{align*}
hence $u\not\in\calF$.

For any $\delta\in(0,1)$, we have
\begin{align*}
&\int_{(1,+\infty)}\lambda^{\delta}\md(E_\lambda u,u)=\sum_{k=0}^\infty\int_{(2^k,2^{k+1}]}\lambda^\delta\md(E_\lambda u,u)\le\sum_{k=0}^\infty\int_{(2^k,2^{k+1}]}2^{\delta(k+1)}\md(E_\lambda u,u)\\
&=\sum_{k=0}^\infty2^{\delta(k+1)}(E((2^k,2^{k+1}])u,u)=\sum_{k=0}^\infty2^{\delta(k+1)}(u_k,u)=\sum_{k=0}^\infty2^{\delta(k+1)}\mynorm{u_k}^2_{L^2(K;\nu)}\\
&\le\sum_{k=0}^\infty2^{\delta(k+1)}\frac{1}{2^{k+1}}=\sum_{k=0}^\infty\frac{1}{2^{(1-\delta)(k+1)}}<+\infty,
\end{align*}
hence by Corollary \ref{cor_domain} (1), we have $u\in\calF^{(\delta)}$ for any $\delta\in(0,1)$.

Therefore, we have $u\in\cap_{\delta\in(0,1)}\calF^{(\delta)}\backslash\calF$.
\end{proof}

\section{Proof of Proposition \ref{prop_hk}}\label{sec_hk}

By Theorem \ref{thm_subordinate}, we only need to show that the generator $A$ of $(\calE,\calF)$ on $L^2(K;\nu)$ is unbounded and that $\calF^{(\delta)}=\calF_{\delta\beta_0}$ for any $\delta\in(0,1)$.

First, we have the following result.

\begin{mylem}\label{lem_unbounded}
Let $(K,d,\nu)$ be an $\alpha$-regular metric measure space. Let $(\calE,\calF)$ be a conservative regular Dirichlet form on $L^2(K;\nu)$ with a heat kernel $p_t(x,y)$ satisfying
\begin{equation}\label{eqn_lower}
p_t(x,y)\ge\frac{C_1}{t^{\alpha/\beta_0}}\Phi\left(C_2\frac{d(x,y)}{t^{1/\beta_0}}\right)
\end{equation}
for any $x,y\in K$, for any $t\in(0,\mathrm{diam}(K)^{\beta_0})$, here $\mathrm{diam}(K)=\sup\myset{d(x,y):x,y\in K}$ is infinite if $K$ is unbounded and is finite if $K$ is bounded, where $\beta_0\in(0,+\infty)$ is some parameter, $C_1,C_2$ are some positive constants and $\Phi:(0,+\infty)\to(0,+\infty)$ is some monotone decreasing function. Then the generator $A$ of $(\calE,\calF)$ on $L^2(K;\nu)$ is unbounded.
\end{mylem}

\begin{myrmk}
The boundedness of the generator is sensitive to the small scale behaviours of the space and the heat kernel. For example, on $\mathbb{Z}$, the generator of the standard random walk is bounded with spectrum $[-2,0]$, but on $\R$, the generator of the standard Brownian motion is unbounded with spectrum $(-\infty,0]$.
\end{myrmk}

\begin{proof}
Since $(\calE,\calF)$ on $L^2(K;\nu)$ is conservative, we have
$$\int_Kp_t(x,y)\nu(\md y)=1\text{ for any }t\in(0,+\infty),\text{ for any }x\in K.$$
Then for any $u\in L^2(K;\nu)$, we have
$$\calE(u,u)=\lim_{t\downarrow0}\frac{1}{t}(u-T_tu,u)=\lim_{t\downarrow0}\frac{1}{2t}\int_K\int_{K}(u(x)-u(y))^2p_t(x,y)\nu(\md y)\nu(\md x),$$
where $t\mapsto\frac{1}{2t}\int_K\int_{K}\ldots\nu(\md y)\nu(\md x)$ is monotone decreasing in $(0,+\infty)$. Hence for any $r\in(0,\mathrm{diam}(K))$, letting $t=r^{\beta_0}$, we have
\begin{align*}
\calE(u,u)&\ge\frac{1}{2t}\int_K\int_{B(x,r)}(u(x)-u(y))^2p_t(x,y)\nu(\md y)\nu(\md x)\\
&\ge\frac{1}{2t}\int_K\int_{B(x,r)}(u(x)-u(y))^2\frac{C_1}{t^{\alpha/\beta_0}}\Phi\left(C_2\frac{d(x,y)}{t^{1/\beta_0}}\right)\nu(\md y)\nu(\md x)\\
&\ge\frac{1}{2t}\int_K\int_{B(x,r)}(u(x)-u(y))^2\frac{C_1}{t^{\alpha/\beta_0}}\Phi\left(C_2\frac{r}{t^{1/\beta_0}}\right)\nu(\md y)\nu(\md x)\\
&=\frac{C_1\Phi(C_2)}{2}\frac{1}{r^{\alpha+\beta_0}}\int_K\int_{B(x,r)}(u(x)-u(y))^2\nu(\md y)\nu(\md x).
\end{align*}

Suppose that $A$ is bounded, then $\sqrt{-A}$ is also bounded, hence $\calF=\calD(\sqrt{-A})=L^2(K;\nu)$ and
$$\calE(u,u)=(\sqrt{-A}u,\sqrt{-A}u)\le\mynorm{\sqrt{-A}}^2\mynorm{u}_{L^2(K;\nu)}^2=\mynorm{A}\mynorm{u}_{L^2(K;\nu)}^2\text{ for any }u\in L^2(K;\nu).$$
Hence
\begin{align}
&\frac{C_1\Phi(C_2)}{2}\frac{1}{r^{\alpha+\beta_0}}\int_K\int_{B(x,r)}(u(x)-u(y))^2\nu(\md y)\nu(\md x)\nonumber\\
&\le\mynorm{A}\mynorm{u}_{L^2(K;\nu)}^2\text{ for any }u\in L^2(K;\nu),\text{ for any }r\in (0,\mathrm{diam}(K)).\label{eqn_bounded}
\end{align}

Since $K$ is $\alpha$-regular, there exists some positive constant $C$ such that
$$\frac{1}{C}r^{\alpha}\le\nu(B(x,r))\le Cr^{\alpha}\text{ for any }x\in K,\text{ for any }r\in(0,\mathrm{diam}(K)).$$

Let $c:=2C^{2/\alpha}\ge2$. We claim that
$$B(x,cr)\backslash B(x,r)\ne\emptyset\text{ for any }x\in K,\text{ for any }r\in(0,\frac{1}{c}\mathrm{diam}(K)).$$
Indeed, we have
$$\nu(B(x,cr))\ge\frac{1}{C}(cr)^\alpha=2^\alpha Cr^{\alpha}>Cr^{\alpha}\ge\nu(B(x,r)),$$
hence $B(x,cr)\backslash B(x,r)\ne\emptyset$.

For any $r\in(0,\mathrm{diam}(K)/(c^3+2))$, take arbitrary $x_0\in K$, let $u=1_{B(x_0,r)}$, then
$$\text{RHS of Equation (\ref{eqn_bounded})}=\mynorm{A}\nu(B(x_0,r))\lesssim r^{\alpha}.$$
On the other hand
\begin{align*}
\text{LHS of Equation (\ref{eqn_bounded})}&\gtrsim\frac{1}{((c^3+2)r)^{\alpha+\beta_0}}\int_K\int_{B(x,(c^3+2)r)}(u(x)-u(y))^2\nu(\md y)\nu(\md x)\\
&\asymp\frac{1}{r^{\alpha+\beta_0}}\int_{B(x_0,r)}\int_{B(x,(c^3+2)r)\backslash B(x_0,r)}(1-0)^2\nu(\md y)\nu(\md x)\\
&=\frac{1}{r^{\alpha+\beta_0}}\int_{B(x_0,r)}\nu({B(x,(c^3+2)r)\backslash B(x_0,r)})\nu(\md x)\\
&\ge\frac{1}{r^{\alpha+\beta_0}}\int_{B(x_0,r)}\nu({B(x_0,c^3r)\backslash B(x_0,r)})\nu(\md x)\\
&=\frac{1}{r^{\alpha+\beta_0}}\nu({B(x_0,c^3r)\backslash B(x_0,r)})\nu(B(x_0,r)).
\end{align*}
Since $B(x_0,c^2r)\backslash B(x_0,cr)\ne\emptyset$, taking arbitrary $z\in B(x_0,c^2r)\backslash B(x_0,cr)$, we have
$$B(z,r)\subseteq B(x_0,c^3r)\backslash B(x_0,r),$$
then
$$\nu(B(x_0,c^3r)\backslash B(x_0,r))\ge\nu(B(z,r)).$$
Hence
$$\text{LHS of Equation (\ref{eqn_bounded})}\gtrsim\frac{1}{r^{\alpha+\beta_0}}\nu(B(z,r))\nu(B(x_0,r))\gtrsim\frac{1}{r^{\alpha+\beta_0}}\cdot r^{\alpha}\cdot r^{\alpha}=r^{\alpha-\beta_0}.$$

Therefore, we have
$$r^{\alpha-\beta_0}\lesssim r^{\alpha}\text{ for any }r\in(0,\frac{\mathrm{diam}(K)}{c^3+2}).$$
Letting $r\downarrow0$, we obtain a contradiction! Hence $A$ is unbounded.
\end{proof}

Second, we show that $\calF^{(\delta)}=\calF_{\delta\beta_0}$ for any $\delta\in(0,1)$. Indeed, the calculation of jumping kernels from heat kernels by subordination is standard. We give the calculation here for completeness.

\begin{proof}[Proof of $\calF^{(\delta)}=\calF_{\delta\beta_0}$ for any $\delta\in(0,1)$]
By Corollary \ref{cor_BD}, we have
\begin{align*}
&\calE^{(\delta)}(u,u)=\frac{1}{2}\int_K\int_K(u(x)-u(y))^2J^{(\delta)}(x,y)\nu(\md x)\nu(\md y),\\
&\calF^{(\delta)}=\myset{u\in L^2(K;\nu):\int_K\int_K(u(x)-u(y))^2J^{(\delta)}(x,y)\nu(\md x)\nu(\md y)<+\infty},
\end{align*}
where
$$J^{(\delta)}(x,y)=\frac{\delta}{\Gamma(1-\delta)}\int_0^{+\infty}\frac{1}{t^{1+\delta}}p_t(x,y)\md t.$$
Hence, we only need to show that
$$J^{(\delta)}(x,y)\asymp\frac{1}{d(x,y)^{\alpha+\delta\beta_0}}.$$

By \cite[Theorem 4.1]{GK08}, we have the following dichotomy.
\begin{enumerate}[(a)]
\item Either $(\calE,\calF)$ on $L^2(K;\nu)$ is local, $\beta_0\in[2,\alpha+1]$ and
$$\Phi(s)\asymp C\exp\left(-cs^{\frac{\beta_0}{\beta_0-1}}\right).$$
\item Or $(\calE,\calF)$ on $L^2(K;\nu)$ is non-local, $\beta_0\in(0,\alpha+1]$ and
$$\Phi(s)\asymp (1+s)^{-(\alpha+\beta_0)}.$$
\end{enumerate}

For (a). We have
$$\frac{C_1}{t^{\alpha/\beta_0}}\exp\left(-C_2\left(\frac{d(x,y)}{t^{1/\beta_0}}\right)^{\frac{\beta_0}{\beta_0-1}}\right)\le p_t(x,y)\le\frac{C_3}{t^{\alpha/\beta_0}}\exp\left(-C_4\left(\frac{d(x,y)}{t^{1/\beta_0}}\right)^{\frac{\beta_0}{\beta_0-1}}\right)$$
for any $x,y\in K$, for any $t\in(0,\mathrm{diam}(K)^{\beta_0})$.

Note the following elementary results. For any $a\in(1,+\infty)$, $b,c,d\in(0,+\infty)$, we have
\begin{align}
\int_0^{+\infty}\frac{1}{t^a}\exp\left(-\frac{c}{t^b}\right)\md t&=\frac{\Gamma\left(\frac{a-1}{b}\right)}{bc^{\frac{a-1}{b}}},\label{eqn_ele1}\\
\int_0^{d}\frac{1}{t^a}\exp\left(-\frac{c}{t^b}\right)\md t&=\frac{1}{bc^{\frac{a-1}{b}}}\int_{\frac{c}{d^b}}^{+\infty}s^{\frac{a-1}{b}-1}e^{-s}\md s.\label{eqn_ele2}
\end{align}

If $\mathrm{diam}(K)=+\infty$. By Equation (\ref{eqn_ele1}), we have
\begin{align*}
&\delta(1-\delta)\left[(\beta_0-1)\frac{\Gamma\left(\frac{(\beta_0-1)(\alpha+\delta\beta_0)}{\beta_0}\right)}{\Gamma(2-\delta)}\frac{C_1}{C_2^{\frac{(\beta_0-1)(\alpha+\delta\beta_0)}{\beta_0}}}\right]\frac{1}{d(x,y)^{\alpha+\delta\beta_0}}\le J^{(\delta)}(x,y)\\
&\le\delta(1-\delta)\left[(\beta_0-1)\frac{\Gamma\left(\frac{(\beta_0-1)(\alpha+\delta\beta_0)}{\beta_0}\right)}{\Gamma(2-\delta)}\frac{C_3}{C_4^{\frac{(\beta_0-1)(\alpha+\delta\beta_0)}{\beta_0}}}\right]\frac{1}{d(x,y)^{\alpha+\delta\beta_0}}\text{ for any }x,y\in K.
\end{align*}

If $\mathrm{diam}(K)<+\infty$. Using semi-group property, we have
$$p_t(x,y)\le\frac{C_3}{\mathrm{diam}(K)^{\alpha}}\text{ for any }t\in[\mathrm{diam}(K)^{\beta_0},+\infty),\text{ for any }x,y\in K.$$
By Equation (\ref{eqn_ele1}), we have
\begin{align*}
&J^{(\delta)}(x,y)=\frac{\delta}{\Gamma(1-\delta)}\left(\int_0^{\mathrm{diam}(K)^{\beta_0}}+\int_{\mathrm{diam}(K)^{\beta_0}}^{+\infty}\right)\frac{1}{t^{1+\delta}}p_t(x,y)\md t\\
&\le\frac{\delta}{\Gamma(1-\delta)}\left(\int_0^{\mathrm{diam}(K)^{\beta_0}}\frac{1}{t^{1+\delta}}\frac{C_3}{t^{\alpha/\beta_0}}\exp\left(-C_4\left(\frac{d(x,y)}{t^{1/\beta_0}}\right)^{\frac{\beta_0}{\beta_0-1}}\right)\md t\right.\\
&\left.+\int_{\mathrm{diam}(K)^{\beta_0}}^{+\infty}\frac{1}{t^{1+\delta}}\frac{C_3}{\mathrm{diam}(K)^\alpha}\md t\right)\\
&\le\frac{\delta}{\Gamma(1-\delta)}\left(\int_0^{+\infty}\frac{1}{t^{1+\delta}}\frac{C_3}{t^{\alpha/\beta_0}}\exp\left(-C_4\left(\frac{d(x,y)}{t^{1/\beta_0}}\right)^{\frac{\beta_0}{\beta_0-1}}\right)\md t+\frac{1}{\delta}C_3\frac{1}{\mathrm{diam}(K)^{\alpha+\delta\beta_0}}\right)\\
&\le(1-\delta)\left[\frac{C_3}{\Gamma(2-\delta)}\left({\delta(\beta_0-1)}\frac{\Gamma\left(\frac{(\beta_0-1)(\alpha+\delta\beta_0)}{\beta_0}\right)}{C_4^{\frac{(\beta_0-1)(\alpha+\beta_0)}{\beta_0}}}+1\right)\right]\frac{1}{d(x,y)^{\alpha+\delta\beta_0}}.
\end{align*}
On the other hand, by Equation (\ref{eqn_ele2}), we have
\begin{align*}
&J^{(\delta)}(x,y)\ge\frac{\delta}{\Gamma(1-\delta)}\int_0^{\mathrm{diam}(K)^{\beta_0}}\frac{1}{t^{1+\delta}}p_t(x,y)\md t\\
&\ge\frac{\delta}{\Gamma(1-\delta)}\int_0^{\mathrm{diam}(K)^{\beta_0}}\frac{1}{t^{1+\delta}}\frac{C_1}{t^{\alpha/\beta_0}}\exp\left(-C_2\left(\frac{d(x,y)}{t^{1/\beta_0}}\right)^{\frac{\beta_0}{\beta_0-1}}\right)\md t\\
&=\frac{\delta}{\Gamma(1-\delta)}C_1\frac{\beta_0-1}{C_2^{\frac{(\beta_0-1)(\alpha+\delta\beta_0)}{\beta_0}}d(x,y)^{\alpha+\delta\beta_0}}\int_{C_2\left(\frac{d(x,y)}{\mathrm{diam}(K)}\right)^{\frac{\beta_0}{\beta_0-1}}}^{+\infty}t^{\frac{(\beta_0-1)(\alpha+\delta\beta_0)}{\beta_0}-1}e^{-t}\md t\\
&\ge\frac{\delta}{\Gamma(1-\delta)}C_1\frac{\beta_0-1}{C_2^{\frac{(\beta_0-1)(\alpha+\delta\beta_0)}{\beta_0}}d(x,y)^{\alpha+\delta\beta_0}}\int_{C_2}^{+\infty}t^{\frac{(\beta_0-1)(\alpha+\delta\beta_0)}{\beta_0}-1}e^{-t}\md t\\
&=\delta(1-\delta)\left[\frac{C_1}{\Gamma(2-\delta)}\frac{\beta_0-1}{C_2^{\frac{(\beta_0-1)(\alpha+\delta\beta_0)}{\beta_0}}}\int_{C_2}^{+\infty}t^{\frac{(\beta_0-1)(\alpha+\delta\beta_0)}{\beta_0}-1}e^{-t}\md t\right]\frac{1}{d(x,y)^{\alpha+\delta\beta_0}}.
\end{align*}

For (b). We have
$$\frac{C_1}{t^{\alpha/\beta_0}}\left(1+\frac{d(x,y)}{t^{1/\beta_0}}\right)^{-(\alpha+\beta_0)}\le p_t(x,y)\le\frac{C_2}{t^{\alpha/\beta_0}}\left(1+\frac{d(x,y)}{t^{1/\beta_0}}\right)^{-(\alpha+\beta_0)}$$
for any $x,y\in K$, for any $t\in(0,\mathrm{diam}(K)^{\beta_0})$.

Since
$$\frac{1}{2^{\alpha+\beta_0}}\left(\frac{1}{t^{\alpha/\beta_0}}\wedge\frac{t}{d(x,y)^{\alpha+\beta_0}}\right)\le\frac{1}{t^{\alpha/\beta_0}}\left(1+\frac{d(x,y)}{t^{1/\beta_0}}\right)^{-(\alpha+\beta_0)}\le \left(\frac{1}{t^{\alpha/\beta_0}}\wedge\frac{t}{d(x,y)^{\alpha+\beta_0}}\right)$$
for any $x,y\in K$, for any $t\in(0,+\infty)$, we may assume that
$$C_1\left(\frac{1}{t^{\alpha/\beta_0}}\wedge\frac{t}{d(x,y)^{\alpha+\beta_0}}\right)\le p_t(x,y)\le C_2\left(\frac{1}{t^{\alpha/\beta_0}}\wedge\frac{t}{d(x,y)^{\alpha+\beta_0}}\right)$$
for any $x,y\in K$, for any $t\in(0,\mathrm{diam}(K)^{\beta_0})$.

If $\mathrm{diam}(K)=+\infty$. Since
$$\int_0^{+\infty}\frac{1}{t^{1+\delta}}\left(\frac{1}{t^{\alpha/\beta_0}}\wedge\frac{t}{d(x,y)^{\alpha+\beta_0}}\right)\md t=\left(\frac{1}{1-\delta}+\frac{\beta_0}{\alpha+\delta\beta_0}\right)\frac{1}{d(x,y)^{\alpha+\delta\beta_0}},$$
we have
\begin{align*}
&\delta\left[\frac{C_1}{\Gamma(2-\delta)}\left(1+\frac{(1-\delta)\beta_0}{\alpha+\delta\beta_0}\right)\right]\frac{1}{d(x,y)^{\alpha+\delta\beta_0}}\le J^{(\delta)}(x,y)\\
&\le\delta\left[\frac{C_2}{\Gamma(2-\delta)}\left(1+\frac{(1-\delta)\beta_0}{\alpha+\delta\beta_0}\right)\right]\frac{1}{d(x,y)^{\alpha+\delta\beta_0}}.
\end{align*}

If $\mathrm{diam}(K)<+\infty$. Using semi-group property, we have
$$p_t(x,y)\le \frac{C_2}{\mathrm{diam}(K)^\alpha}\text{ for any }t\in[\mathrm{diam}(K)^{\beta_0},+\infty),\text{ for any }x,y\in K.$$
Hence
\begin{align*}
&J^{(\delta)}(x,y)=\frac{\delta}{\Gamma(1-\delta)}\left(\int_0^{\mathrm{diam}(K)^{\beta_0}}+\int_{\mathrm{diam}(K)^{\beta_0}}^{+\infty}\right)\frac{1}{t^{1+\delta}}p_t(x,y)\md t\\
&\le\frac{\delta}{\Gamma(1-\delta)}\left(\int_0^{\mathrm{diam}(K)^{\beta_0}}\frac{1}{t^{1+\delta}}C_2\left(\frac{1}{t^{\alpha/\beta_0}}\wedge\frac{t}{d(x,y)^{\alpha+\beta_0}}\right)\md t\right.\\
&\left.+\int_{\mathrm{diam}(K)^{\beta_0}}^{+\infty}\frac{1}{t^{1+\delta}}\frac{C_2}{\mathrm{diam}(K)^{\alpha}}\md t\right)\\
&\le\frac{\delta}{\Gamma(1-\delta)}\left(\int_0^{+\infty}\frac{1}{t^{1+\delta}}C_2\left(\frac{1}{t^{\alpha/\beta_0}}\wedge\frac{t}{d(x,y)^{\alpha+\beta_0}}\right)\md t+\int_{\mathrm{diam}(K)^{\beta_0}}^{+\infty}\frac{1}{t^{1+\delta}}\frac{C_2}{\mathrm{diam}(K)^{\alpha}}\md t\right)\\
&=\frac{\delta}{\Gamma(1-\delta)}C_2\left(\left(\frac{1}{1-\delta}+\frac{\beta_0}{\alpha+\delta\beta_0}\right)\frac{1}{d(x,y)^{\alpha+\delta\beta_0}}+\frac{1}{\delta}\frac{1}{\mathrm{diam}(K)^{\alpha+\delta\beta_0}}\right)\\
&\le\delta\left[\frac{C_2}{\Gamma(2-\delta)}\left(1+\frac{(1-\delta)\beta_0}{\alpha+\delta\beta_0}+\frac{1-\delta}{\delta}\right)\right]\frac{1}{d(x,y)^{\alpha+\delta\beta_0}}.
\end{align*}
On the other hand
\begin{align*}
&J^{(\delta)}(x,y)\ge\frac{\delta}{\Gamma(1-\delta)}\int_0^{\mathrm{diam}(K)^{\beta_0}}\frac{1}{t^{1+\delta}}p_t(x,y)\md t\\
&\ge\frac{\delta}{\Gamma(1-\delta)}\int_0^{\mathrm{diam}(K)^{\beta_0}}\frac{1}{t^{1+\delta}}C_1\left(\frac{1}{t^{\alpha/\beta_0}}\wedge\frac{t}{d(x,y)^{\alpha+\beta_0}}\right)\md t\\
&\ge\frac{\delta}{\Gamma(1-\delta)}\int_0^{d(x,y)^{\beta_0}}\frac{1}{t^{1+\delta}}C_1\left(\frac{1}{t^{\alpha/\beta_0}}\wedge\frac{t}{d(x,y)^{\alpha+\beta_0}}\right)\md t\\
&=\frac{\delta}{\Gamma(1-\delta)}\int_0^{d(x,y)^{\beta_0}}\frac{1}{t^{1+\delta}}C_1\frac{t}{d(x,y)^{\alpha+\beta_0}}\md t\\
&=\delta\frac{C_1}{\Gamma(2-\delta)}\frac{1}{d(x,y)^{\alpha+\delta\beta_0}}.
\end{align*}
\end{proof}

\bibliographystyle{plain}

\begin{thebibliography}{10}

\bibitem{BB89}
Martin~T. Barlow and Richard~F. Bass.
\newblock The construction of {B}rownian motion on the {S}ierpi\'nski carpet.
\newblock {\em Ann. Inst. H. Poincar\'e Probab. Statist.}, 25(3):225--257,
  1989.

\bibitem{BB90}
Martin~T. Barlow and Richard~F. Bass.
\newblock On the resistance of the {S}ierpi\'nski carpet.
\newblock {\em Proc. Roy. Soc. London Ser. A}, 431(1882):345--360, 1990.

\bibitem{BBS90}
Martin~T. Barlow, Richard~F. Bass, and John~D. Sherwood.
\newblock Resistance and spectral dimension of {S}ierpi\'nski carpets.
\newblock {\em J. Phys. A}, 23(6):L253--L258, 1990.

\bibitem{BP88}
Martin~T. Barlow and Edwin~A. Perkins.
\newblock Brownian motion on the {S}ierpi\'nski gasket.
\newblock {\em Probab. Theory Related Fields}, 79(4):543--623, 1988.

\bibitem{Ber96}
Jean Bertoin.
\newblock {\em L\'{e}vy processes}, volume 121 of {\em Cambridge Tracts in
  Mathematics}.
\newblock Cambridge University Press, Cambridge, 1996.

\bibitem{FOT11}
Masatoshi Fukushima, Yoichi Oshima, and Masayoshi Takeda.
\newblock {\em Dirichlet forms and symmetric Markov processes}, volume~19 of
  {\em De Gruyter studies in mathematics ; 19}.
\newblock de Gruyter, Berlin [u.a.], 2., rev. and extended ed. edition, 2011.

\bibitem{GK08}
Alexander Grigor'yan and Takashi Kumagai.
\newblock On the dichotomy in the heat kernel two sided estimates.
\newblock In {\em Analysis on graphs and its applications}, volume~77 of {\em
  Proc. Sympos. Pure Math.}, pages 199--210. Amer. Math. Soc., Providence, RI,
  2008.

\bibitem{GY18}
Alexander Grigor'yan and Meng Yang.
\newblock Determination of the walk dimension of the sierpi\'nski gasket
  without using diffusion.
\newblock {\em J. Fractal Geom.}, 5(4):419--460, 2018.

\bibitem{GY19}
Alexander Grigor'yan and Meng Yang.
\newblock Local and non-local {D}irichlet forms on the {S}ierpi\'{n}ski carpet.
\newblock {\em Trans. Amer. Math. Soc.}, 372(6):3985--4030, 2019.

\bibitem{Kig89}
Jun Kigami.
\newblock A harmonic calculus on the {S}ierpi\'nski spaces.
\newblock {\em Japan J. Appl. Math.}, 6(2):259--290, 1989.

\bibitem{Kig93}
Jun Kigami.
\newblock Harmonic calculus on p.c.f.\ self-similar sets.
\newblock {\em Trans. Amer. Math. Soc.}, 335(2):721--755, 1993.

\bibitem{Kig01}
Jun Kigami.
\newblock {\em Analysis on fractals}, volume 143 of {\em Cambridge Tracts in
  Mathematics}.
\newblock Cambridge University Press, Cambridge, 2001.

\bibitem{KZ92}
Shigeo Kusuoka and Xian~Yin Zhou.
\newblock Dirichlet forms on fractals: {P}oincar\'e constant and resistance.
\newblock {\em Probab. Theory Related Fields}, 93(2):169--196, 1992.

\bibitem{McG02}
Ivor McGillivray.
\newblock Resistance in higher-dimensional {S}ierpi\'nski carpets.
\newblock {\em Potential Anal.}, 16(3):289--303, 2002.

\bibitem{Pie08}
Katarzyna Pietruska-Pa{\l}uba.
\newblock Limiting behaviour of {D}irichlet forms for stable processes on
  metric spaces.
\newblock {\em Bull. Pol. Acad. Sci. Math.}, 56(3-4):257--266, 2008.

\bibitem{Sat99}
Ken-iti Sato.
\newblock {\em L\'{e}vy processes and infinitely divisible distributions},
  volume~68 of {\em Cambridge Studies in Advanced Mathematics}.
\newblock Cambridge University Press, Cambridge, 1999.
\newblock Translated from the 1990 Japanese original, Revised by the author.

\bibitem{Yan17}
Meng {Yang}.
\newblock {Construction of Local Regular Dirichlet Form on the Sierpi{\'n}ski
  Gasket using $\Gamma$-Convergence}.
\newblock {\em arXiv e-prints}, page arXiv:1706.04998, Jun 2017.

\bibitem{Yan18}
Meng Yang.
\newblock Equivalent semi-norms of non-local {D}irichlet forms on the
  {S}ierpi\'{n}ski gasket and applications.
\newblock {\em Potential Anal.}, 49(2):287--308, 2018.

\end{thebibliography}

\def\cprime{$'$}

\end{document}